\documentclass[11pt,preprint,3p,authoryear,a4paper]{elsarticle}
\usepackage{graphicx}
\usepackage{xcolor}
\usepackage{dsfont}
\usepackage{amsmath}
\usepackage{amssymb}
\usepackage{amsfonts}
\usepackage{amsbsy} 
\usepackage{amsthm}
\usepackage{array}
\usepackage{booktabs} 
\usepackage{verbatim} 
\usepackage[english]{babel}
\usepackage{float}
\usepackage{epsfig}
\usepackage{pdfsync}
\usepackage{verbatim}
\usepackage{mathrsfs} 
\usepackage{subcaption}
\usepackage{ulem}
\usepackage{pdfpages}
\usepackage{varioref}
\usepackage{setspace}
\usepackage{makecell}
\usepackage{hyperref}

	\newcounter{arclist}

		\newcounter{arcenum}
		\newenvironment{arcenum}{
			\begin{list}{\arabic{arcenum}.}{
					\usecounter{arcenum}
					\topsep=1pt
					\itemsep=1pt 
					\parsep=0pt 
					\leftmargin=19pt 
				}}
				{\end{list}}

\newcolumntype{C}[1]{>{\centering\let\newline\\\arraybackslash\hspace{0pt}}m{#1}}
\renewcommand\appendix{\par
\setcounter{section}{0}%
\setcounter{subsection}{0}%
\setcounter{table}{0}
\setcounter{table}{0}
\setcounter{figure}{0}
\gdef\thetable{\Alph{table}}
\gdef\thefigure{\Alph{figure}}
\gdef\thesection{\Alph{section}}
\setcounter{section}{0}}

\usepackage{titlesec}

\titleformat{\paragraph}
{\normalfont\normalsize\itshape}{\theparagraph}{1em}{}
\titlespacing*{\paragraph}
{0pt}{3.25ex plus 1ex minus .2ex}{1.5ex plus .2ex}

\newcommand{\rev}[1]{\textcolor{black}{#1}}

\newtheorem{theorem}{Theorem}[section]

\newtheorem{remark}{Remark}[section]

\newcommand{\HL}[1]{\textcolor{red}{[#1]\,}}

\newcommand{\parD}[1]{\frac{\partial}{\partial #1}}

\begin{document}


\normalem 



\begin{frontmatter}

\title{Optimal reinsurance design under solvency constraints}

\author[UMelb]{Benjamin Avanzi\corref{cor}}
\ead{b.avanzi@unimelb.edu.au}

\author[]{Hayden Lau}
\ead{kawai.lau90@gmail.com}

\author[Cop]{Mogens Steffensen}
\ead{mogens@math.ku.dk}

\cortext[cor]{Corresponding author.}

\address[UMelb]{Centre for Actuarial Studies, Department of Economics, University of Melbourne VIC 3010, Australia}
\address[Cop]{Department of Mathematical Sciences, University of Copenhagen, DK-2100 Copenhagen, Denmark}

\begin{abstract}

We consider the optimal \rev{risk transfer from an insurance company to a reinsurer}. The problem formulation considered in this paper is closely connected to the optimal portfolio problem in finance, with some crucial distinctions. In particular, the insurance company's surplus is here (as is routinely the case) approximated by a Brownian motion, as opposed to the geometric Brownian motion used to model assets in finance. Furthermore, risk exposure is dialled ``down'' via reinsurance, rather than ``up'' via risky investments. This leads to interesting qualitative differences in the optimal designs.

In this paper, using the martingale method, we derive the optimal design as a function of proportional, non-cheap reinsurance design that maximises the quadratic utility of the terminal value of the insurance surplus. We also consider several realistic constraints on the terminal value: a strict lower boundary, the probability (Value at Risk) constraint, and the expected shortfall (conditional Value at Risk) constraints under the $\mathbb{P}$ and $\mathbb{Q}$ measures, respectively. In all cases, the optimal reinsurance designs boil down to a combination of proportional protection and option-like protection (stop-loss) of the residual proportion with various deductibles. Proportions and deductibles are set such that the initial capital is fully allocated. Comparison of the optimal designs with the optimal portfolios in finance is particularly interesting. Results are illustrated.

\end{abstract}


\end{frontmatter}

\numberwithin{equation}{section}

\section{Introduction}

\subsection{Background and motivation}

The choice of reinsurance cover is a perennial problem in insurance \citep*[see, e.g.][]{AlBeTe17}. Not only have insurers always sought to control their risk level through purchasing well-designed reinsurance contracts, but this is now typically mandated via setting and implementing a chosen \emph{risk appetite}. 

In insurance, early optimality results about the choice of reinsurance include the optimality of stop-loss reinsurance when variance is minimised, and premiums are calculated using an expected value principle \citep*[see, for instance, the seminal work of][]{Arr63}. More recently, \citet*{BaBaHe09,BaBaHe11} considered the stability of optimal reinsurance designs for a range of risk measures, in a discrete probability space. \citet*{CaCh20} provide a recent review of optimal reinsurance designs based on risk measures and comment on the close connection between the insurance and finance fields. 

Indeed, the optimal reinsurance problem is closely connected to the optimal portfolio problem in finance (with some important distinctions discussed below). This connection appears often naturally. The immunisation theory, often credited to British actuary \citet*{Red52}, is a famous example. A recent example of straddling financial and actuarial mathematics in a reinsurance context includes \citet*{BaBaBa21}, who consider the omega ratios in financial and actuarial applications. Reinsurance is optimised with respect to the ceding company’s omega ratio.

The optimal portfolio literature in finance is, of course, also very developed. The so-called \emph{Merton's portfolio problem} \citep*{Mer69} is a classical topic, which keeps being revisited and extended. This problem typically includes maximising the utility of a terminal value and consumption in-between; see \citet*{Kor97} for a comprehensive treatment of its classical formulation.

A major distinction between the insurance and finance optimal control literature is in the nature of the underlying stochastic process. While finance models typically consider multiplicative models such as the geometric Brownian motion, insurance (surplus) models take an absolute approach. The classical insurance model---the Cram\'er-Lundberg model \citep*{Cra55,Lun09}---involves deterministic income and compound Poisson claim payments, but is sometimes approximated by a (non-geometric) diffusion process \citep*[see, for instance,][and references therein]{Bau04, Bau04b,Sch08,OkSu10,TaWeWeZh20,BaZhXiGaZh22}.

Another major distinction is in the nature of the criteria used to design contracts. In ``traditional'' actuarial mathematics, \emph{stability criteria} \citep{Buh70,Ger72} are used to control the surplus processes to make decisions about risk and the design of reinsurance; these include the probability of ruin \citep{Cra55,Lun09,AsAl10}, and the expected present value of dividends \citep{deF57,Ava09,AlTh09,GuXuZh22}. In this paper, we use a different objective and will borrow the linear-quadratic criterion of finance to formulate our objective function.

The problem we consider in this paper is directly related to the ``optimal risk sharing'' problem, which would typically be formulated in a slightly broader way. We chose to focus on the reinsurance interpretation, as it brings context to our solvency constraints. Nevertheless, our results also contribute to the optimal risk-sharing literature. Specifically, they provide insights into the optimal design of the risk-sharing arrangement, as discussed in the next section.

\begin{remark}
The recent and very interesting area of research of optimal reinsurance and investment, such as considered in \citet*{AsBo18,BaZhXiGaZh22} considers a connected, but different problem from the one considered in this paper. These papers take a game theoretic approach to the optimal reinsurance problem. For instance, the reinsurer is a Stackelberg game leader in \citet*{BaZhXiGaZh22} with two insurer followers. Reinsurers and insurers optimise the expected utility of terminal wealth with delay (and relative performance for the insurers).

\rev{In this paper, the reinsurer sets its premium rate exogenously, and the insurer is optimising some objective for given such rate. Extending the paper along the lines of the literature mentioned above, where the reinsurer could actively control its premium rate, could be possible and would be quite interesting.} 
\end{remark}

\subsection{Statement of contributions} \label{S_Contr}

In this paper, we focus on modifying a well-studied problem in finance to derive optimal reinsurance designs under similar constraints but in an insurance context. In the finance literature, constrained portfolio problems with utility functions often consider the underlying wealth process as a geometric Brownian motion because asset prices cannot become negative \citep*[in papers such as][bankruptcy is defined as a portfolio value of 0, which is only achieved via excess consumption]{KaLeSeSh86}. 

We model the surplus of an insurance company---which can become negative---as a Brownian motion and explore the impacts of a range of constraints on the terminal surplus: (i) a strict lower bound, (ii) Value at Risk, and (iii) Expected Shortfall (under both the $\mathbb{P}$ and $\mathbb{Q}$ measures) constraints. These are well-founded in regulatory requirements for insurance entities, and they make particular sense in the insurance context where the wealth process can become negative. This matters because solvency impacts insurance demand and also optimal reinsurance designs \citep*[see, e.g.][]{EcGa18,ReScSc21}. It is also a major consideration for Enterprise Risk Management and in the regulation of financial institutions (e.g. Solvency II in Europe).

We control the insurance company's surplus via the design of a reinsurance contract formulated in terms of proportional reinsurance. This may at first glance seem like a strong limitation, but it actually is not because we model the surplus as a Brownian motion. While the Brownian motion can be considered an approximation to an underlying classical risk process (e.g. a compound Poisson process), under such a continuous approximation, differences in reinsurance design over an infinitesimal amount of time (which is used to specify the process dynamics) vanish, because any risk transfer would be able to be formulated as a proportion over $dt$. However, we can derive and draw insights from this proportion's optimal ``shape'' (as a function of the information available at that time). Thus, the proportional reinsurance specification is not a restriction. 

Notwithstanding the discussion in the previous paragraph, it arguably is a restriction to work with an approximating process. It is beyond the scope of this presentation to consider the optimal design and objective in the case of an underlying non-diffusive risk process. There, we believe, results may look different.

From a mathematical point of view, we derive the terminal wealth structure that maximises its quadratic utility, subject to the solvency constraints mentioned above. This is in contrast with the traditional insurance-related literature, which uses different criteria \citep[see, e.g.][]{AsAl10,AlTh09,CaLaLi20}. Furthermore, while the insurance literature often  considers objectives formulated on an infinite time horizon \citep[but not always; see][for recent counter examples]{ZhMeZe16,GuXuZh22}, it is really for mathematical convenience, as in reality companies control their risk over a finite time outlook. This would typically be at least one year (for regulation and solvency purposes), but more importantly for strategic (management) decisions 5 to 10 years. For such a time horizon a proportional reinsurance contract may not be optimal. As discussed above, it is relevant to generalize to a more optimal dynamic design with proportional reinsurance as the basis.

Optimality is shown using martingale techniques \citep*[see, e.g.][]{Kor97,BaSh01,GuStZh21}. The optimal terminal wealth design structure is generally a combination of stop-loss (with or without limit) and proportional \emph{protection} on the terminal wealth. The former is operationalised via options in finance, a well-known feature of the Merton problem with constraints.

While strict constraints were considered in the insurance context, a Value at Risk constraint (which sets a maximum probability of failing to achieve a certain level) is only considered for level 0; see \citet*{KoWi08}. Expected Shortfall constraints under the $\mathbb{P}$ and $\mathbb{Q}$ measures were considered in the finance context \citep*[see][for the most recent results]{GuStZh21}, although not in the insurance context (with an additive underlying process).

Furthermore, the finance literature commonly works with power utility functions. This does not work in our context since the surplus process can become negative. We choose here to work with quadratic utility instead (closely connected to the mean-variance formulation; see Section \ref{S_ObjectiveFunction}). In actuarial science, the so-called \emph{linear-quadratic} objective function has appeared in the literature specifically on risk management of pension funds, see \citet*{Ste06} and references therein. However, it has not gained much attention as an objective in controlling the classical surplus process of insurance. A critical connection between classical (concave) utility optimization and linear-quadratic problems is that the linear-quadratic objective can be thought of as a second-order polynomial expansion to an ordinary utility objective. Varying over the centre point of the expansion and expansions to higher orders than two has recently been studied by \citet*{NoSt17} and \citet*{FaSu20}.

Our results can be directly compared with those in finance. They differ in several ways, which is remarkable. In finance, the $\mathbb{Q}$ measure version of the expected shortfall problem is easier to implement because it only involves vanilla options. In contrast, the $\mathbb{P}$ design involves power options \citep*{BaSh01,GuStZh21}. At first glance the problem where the expected loss is measured under the $\mathbb{P}-$expectation seems easier to communicate, but \citet*{GuStZh21} discuss the merits of the problem where the expected loss is measured under the $\mathbb{Q}-$expectation problem from a communication perspective. In this paper, in the insurance context, we show that the solution to the $\mathbb{P}-$expectation problem is the slightly simpler one to implement and that neither the solution to the $\mathbb{P}-$expectation problem nor the solution to the $\mathbb{Q}-$expectation problem involves power option elements (recall that the put option in finance can be seen as a stop loss protection in insurance).

Optimal designs are provided in the form of optimal terminal wealth, which is a direct result of the optimisation methodology we are using; see Section \ref{S_meth}. From these, the optimal reinsurance designs can be directly obtained; see Section \ref{S_OR_prel}, Appendix \ref{ss.comp.pit} and in particular Equation \eqref{sol.pit}.  In all cases, the optimal reinsurance designs boil down to a combination of proportional protection and option-like protection (stop-loss) of the residual proportion with various deductibles. Proportions and deductibles are set such that the initial capital is fully allocated. These can easily be depicted; see Figures \ref{F_Strict}-\ref{F_optishapes} and Remark \ref{R_interpret}.

\begin{remark}
While it is possible to write closed-form mathematical expressions for optimal reinsurance proportions over time, these are not very tidy and hard to interpret, which is why we chose not to include them in the body of the paper.  They also aren't our main focus, as we are more interested in the corresponding optimal reinsurance design as  explained above. Nevertheless, these are implemented in R, with code freely available on GitHub; see Section \ref{S_NI}.
\end{remark}

\subsection{Structure of the paper}

In the next section, we introduce our mathematical framework. We define the insurance surplus dynamics considered in this paper and its admissible reinsurance designs. The objective to maximise, and the possible solvency constraints the problem is subject to, are also formally defined. Section \ref{S_OR} provides the designs (and proofs) to all of our sub-problems, after some methodological preliminaries and explanations. Those designs are illustrated in Section \ref{S_NI}, where optimal design parameters are computed for given constellations of parameters. We draw the corresponding optimal terminal wealth functions and perform simulations to show traces of the optimally controlled surplus process and the associated optimal reinsurance design. Section \ref{S_C} concludes.

\section{Model framework} \label{S_MF}

This section introduces our model framework: risk process, design and objective function. All the solvency constraints considered in the paper are also introduced here.

\subsection{Insurance surplus dynamics and reinsurance design} \label{S_MF_surplus}

We introduce a probability space $(\Omega,\mathscr{F}_0,\mathbb{P})$. We consider the diffusion approximation to be a risk process with dynamics in the form
\begin{equation}
d\widetilde{X}_t=adt+\sigma dW_t,
    \quad \widetilde{X}_0=x<\widetilde{k},
\end{equation}
where $W$ is a standard Brownian motion on the probability space. The attribution of a `tilde' to $X$ is introduced here for notational efficiency as we later define an auxiliary process $X$ from $\tilde{X}$ with a shifted drift; see Section \ref{S_OR_prel}. All problems are represented, and all designs (first) are specified through that auxiliary process without a tilde (without loss of generality).

It is well-known that a compensated compound Poisson claims process becomes a Brownian motion for sufficiently frequent and small claims. Hence, the diffusion approximation is quite standard \citep*[see, for instance,][and references therein, for a proof of the result, examples of applications, and associated discussions]{Bau04,Bau04b,Sch08,OkSu10,TaWeWeZh20,BaZhXiGaZh22}. It should be noted, though, that our results below are derived here exclusively under that approximation, and we can not claim that the structure of the results holds in more generality.

Now, we assume that the risk manager can sell off a proportion of the business through proportional reinsurance. The reinsurance is non-cheap, meaning that the reinsurer demands a risk premium larger than the one demanded first by the insurer. Formally, the reinsurer demands a drift $b$, $b>a$, for taking over the full risk. If we denote by $\pi_t$ the proportion of risk covered by the reinsurer at time $t$ in a specific reinsurance product design, the dynamics of the controlled risk process are then given by
\begin{equation} \label{E_controlledXtilde}
    d\widetilde{X}_t^\pi=(a-b\pi_t)dt+(1-\pi_t)\sigma dW_t,
    \quad \widetilde{X}_0^\pi=x<\widetilde{k}.
\end{equation}
Note that if the insurer sells off all risk, i.e. $\pi_t\equiv 1$, the dynamics become
\begin{equation} \label{E_fullreins}
    d\widetilde{X}_t^1=(a-b)dt,
    \quad \widetilde{X}_0^1=x<\widetilde{k},
\end{equation}
and the insurer loses money with certainty because of the non-cheap full reinsurance protection.
As mentioned in the introduction, we emphasise that it is not a restriction to work with proportional reinsurance as a basis for the design as long as it is applied directly to the diffusion approximation. Such a design translates back to any reinsurance design of a potential underlying non-diffusive process. However, we cannot guarantee that the optimal design we find for the diffusion approximation is also optimal for the underlying non-diffusive state process.

\begin{remark}
\rev{The combination of the diffusion approximation and the reinsurance pricing principle expressed through the premium rate $b$ needs a further remark. Analogously to the fact that our result is robust with respect to the reinsurance design after implementing the diffusion approximation, it is also robust to the reinsurance premium principle across all linear combinations of the expectation and variance premium principles. However, it is not robust towards more general premium principles than those as they would introduce non-linearity of $\pi$ in the drift term of $X$ after reinsurance. However, this is crucially post-approximation robustness and relies on dynamic control. As mentioned above, the optimal control we find may not be optimal in the underlying non-diffusive state process, and optimality may be lost independently of the reinsurance premium principle applied.}
\end{remark}

\subsection{Objective function} \label{S_ObjectiveFunction}

We are interested in studying an objective in quadratic form, i.e. the objective is to minimize the quadratic distance to a certain surplus target at a given time $T$. Thus, we search for
\begin{equation}\label{E_objective}
    \max_{\pi\in\Pi}\mathbb{E}\left[-\frac{1}{2}\left(\widetilde{k}-\widetilde{X}_T^\pi\right)^2\right],
\end{equation}
and the corresponding optimizing design, where $\Pi$  is the set of admissible designs (defined shortly). Let $$\mathscr{F}=\{\mathscr{F}_t=\sigma(W_s,s\in[0,t]),t\geq 0\}$$
be the filtration generated by $W$, and define $\Pi$ by
$$\Pi:=\{\pi:\text{adapted to }\mathscr{F},\forall \omega\in\Omega\backslash N,~ \exists B_{\pi,\omega}>0 \text{ such that }|\pi_t(\omega)|<B_{\pi,\omega},t\in[0,T]\},$$
where $N$ is a $\mathbb{P}$-null set. All our constructed designs are admissible; see \eqref{sol.pit}.

Without further constraints, it may well be that $\pi_t$ is outside the interval $[0,1]$. This means that the insurer can, in principle, act as a reinsurer himself and take over more of the risk than he originally had $(\pi_t<0)$ or pass on more risk than he had, i.e. go short in insurance risk $(\pi_t>1)$. We allow these positions below, disregarding classical financial borrowing and short-selling constraints. Considering optimization under such constraints is beyond the scope of this paper. That being said, from \eqref{E_objective} and \eqref{E_controlledXtilde}, intuitively $\pi_t>1$ is never optimal; indeed, this is observed in our numerical illustrations in Section \ref{S_NI}.

Varying here over $\widetilde{k}$ in \eqref{E_objective} is equivalent to varying over the mean condition in a mean-variance problem of the form
\begin{equation}\label{E_MV}
    \min_{\pi\in\Pi, \mathbb{E}[\widetilde{X}_T^\pi]=p}\text{Var}[\widetilde{X}_T^\pi] \text{ for some } p\in\mathbb{R};
\end{equation}
a well known objective in the finance literature \citep*[see, e.g.,][]{Kor97,Kor05}. Note that this version of a mean-variance objective is based on the so-called pre-commitment approach to a problem with time-consistency issues. For an equilibrium approach to optimal reinsurance under the mean-variance criterion but without solvency constraints, see \citet{LiLiYo17}.

\begin{remark} \label{R_MV}
The following point is crucial when thinking about \eqref{E_objective} as a mean-variance formulation. When forming a dynamic design problem below, we work with the value function
\begin{equation}
    \max_{\pi\in\Pi}\mathbb{E}_{t,x}\left[-\frac{1}{2}\left(\widetilde{k}-\widetilde{X}_T^\pi\right)^2\right],
\end{equation}
where $\mathbb{E}_{t,x}$ denotes the conditional expectation $\mathbb{E}\left[\cdot \middle| X(t)=x\right]$. In the well-known mean-variance framework, the mean condition is \emph{unconditional}. Our context can only be interpreted as a mean-variance problem from time 0, not at subsequent time points. Said differently, we are in the traditional mean-variance case with pre-commitment (to the mean, around which quadratic distance is measured) and not in the modern version where the variance is the conditional variance. In the modern version, time-inconsistency issues arise such that traditional global optimization of the design cannot be dealt with by dynamic programming but has to be replaced by a different notion of optimality, e.g. the one introduced in the equilibrium approach.
\end{remark}

\begin{remark}
The finance literature typically works with the power utility. Unfortunately, this does not work with non-geometric processes because power utility assigns minus infinity utility to negative surplus values. This problem can only be avoided in the presence of strict constraints; it remains for all other constraints considered in this paper.

One could consider exponential utility as an alternative to quadratic utility, but we decided to work with a quadratic formulation, mainly due to its nice interpretation; see Equation \eqref{E_MV}.
\end{remark}

\begin{remark}
Note the assumption $\widetilde{X}_t^\pi\leq \widetilde{k}$ in \eqref{E_controlledXtilde} is reasonable because once $\widetilde{X}_t^\pi= \widetilde{k}$, the optimal design is to take $\pi_s=1$ for $s\in[t,T]$ because of \eqref{E_objective}. In this sense, we see that there are two important levels: $x=\widetilde{k}$ (an upper bound) and $x=\widetilde{C}$ (a lower solvency threshold which may be strict or not).
\end{remark}

\subsection{Overview of the solvency constraints}

Considering a range of solvency constraints is a major contribution of this paper, which makes particular sense in the insurance context, as explained below. We consider five cases:
\begin{arcenum}
\item No constraint, where \eqref{E_objective} is maximised without any constraint on its terminal value (either strict or in a probabilistic sense).
\item Strict constraint, where we must achieve
\begin{equation}\label{E_strictconstr}
\widetilde{X}_T^\pi \ge \widetilde{C}
\end{equation}
with certainty. Here, $\widetilde{C}$ could be considered an absolute minimum solvency level.
\item A relaxation of \eqref{E_strictconstr}, where the inequality must be achieved only with given probability $1-\epsilon$, that is,
\begin{equation} \label{E_probconstr}
\mathbb{P}\left[\widetilde{X}_T^\pi \ge \widetilde{C}\right]\ge 1-\epsilon.    
\end{equation}
Note that this corresponds to a \emph{Value at Risk} type of constraint. 
Of course, we have that \eqref{E_probconstr} becomes \eqref{E_strictconstr} as $\epsilon \rightarrow 0$.
\item Expected shortfall under $\mathbb{P}$, that is,
\begin{equation} \label{E_ESconstrP}
\mathbb{E}\left[(\widetilde{C}-\widetilde{X}_T^\pi)_+ \right]\le \nu.
\end{equation}
This sets an upper bound to the expected excess loss above a certain threshold $\widetilde{C}$. If this threshold is set as a Value at Risk, this constraint is the same as the so-called \emph{conditional Value at Risk} or \emph{Tail Value at Risk}. Such a constraint was shown to possess better properties than the straight Value at Risk \citep[also in finance, see for instance][]{BaSh01}, and superseded the Value at Risk in several regulatory frameworks, such as in Basel III (and subsequent standards).
Of course, we have that \eqref{E_ESconstrP} becomes \eqref{E_strictconstr} as $\nu\rightarrow 0$.
\item Expected shortfall under $\mathbb{Q}$, which is the risk-neutral measure used to present financial values as expectations, defined formally in Section \ref{S_ESQ} (as additional notation, which we have not defined yet, is required). Thus, the shortfall is measured analogously to \eqref{E_ESconstrP} with the crucial difference that the expectation is based on the risk-neutral probabilities. The study of this constraint is motivated by \citet*[in the finance context]{GuStZh21}; see also Section \ref{S_Contr}. It makes sense from a risk management perspective, as it sets an upper threshold to the financial value (price) of the (unbought) protection on the expected shortfall of $\widetilde{X}_t^\pi$ from $\widetilde{C}$ (up to a risk-free discount factor).
We also have here that \eqref{E_ESconstrP} becomes \eqref{E_strictconstr} as $\nu\rightarrow 0$.

\end{arcenum}
Note that the closest actuarial literature that considers such constraints is \citet{KoWi08}. Still, they do not consider items 4 and 5 and only a limited version of item 3 (restricted to $\widetilde{C}\equiv 0$).

\section{Optimality designs} \label{S_OR}

\subsection{Preliminaries} \label{S_OR_prel}

To neutralize the cost of the reinsurance and for convenience, we introduce a set of auxiliary processes and constants. From \eqref{E_fullreins}, we define the auxiliary wealth process
\begin{equation}
    X_t^\pi:=\widetilde{X}_t^\pi-(a-b)t,
\end{equation}
as well as constants,
\begin{align}
    k:=\widetilde{k}-(a-b)T,\\
    C:=\widetilde{C}-(a-b)T.
\end{align}
It is easy to see that we can now replace the risk process and constants $(\widetilde{X},\widetilde{k},\widetilde{C})$ by the auxiliary counterparts $(X,k,C)$, bearing in mind that, in the end, we need to adjust the designs to return to the original versions. Thus, we consider an (equivalent) optimization problem, 
\begin{equation}\label{Problem.U}
    \max_{\pi\in\Pi}\mathbb{E}_{t,x}\left[-\frac{1}{2}\left(k-X_T^\pi\right)^2\right]
\end{equation}
with
\begin{equation}
    dX_t^\pi=(1-\pi_t)(bdt +\sigma dW_t),
    \quad X(0)=x<\widetilde{k}.
\end{equation}
In other words, we now focus purely on the residual risk dynamics left after reinsurance: $dX_t^\pi$ describes the profit generated by retaining business $(1-\pi_t)$ at time $t$, as compared to \emph{full} reinsurance protection (which is loss-making since $b>a$).

We now introduce a second auxiliary process. In the design below, we are making use of a special process $Z=\{Z_t,t\in[0,T]\}$ with dynamics
\begin{equation}
    dZ_t=\beta Z_t dW_t,\quad
    Z_0=1, \quad
    \beta=-\frac{b}{\sigma};
\end{equation}
and solution
\begin{equation}
    Z_t=e^{-\frac{1}{2}\beta^2t+\beta W_t}.
\end{equation}
Note that the highest possible upside from risk $W \rightarrow \infty$ corresponds to $Z\rightarrow 0$ because $\beta<0$. Furthermore, it is easy to see that $XZ$ forms a martingale such that, for a given terminal surplus, the path $X_t^\pi$ can be represented by
\begin{equation}\label{find.sample.path}
    X_t^\pi=\mathbb{E}\left[\left.\frac{Z_T}{Z_t}X_T^\pi\right|\mathscr{F}_t\right].
\end{equation}
Note that \eqref{find.sample.path} is not a definition of $X_t$ but a convenient representation relying on former definitions. Thus, once we know $X_T^\pi$, we can evaluate this expression to calculate $X_t$ and thereafter  apply \rev{It\=o's Lemma} to deduce the optimal proportion $\pi_t$; see also Appendix \ref{ss.comp.pit}.

\subsection{Methodology} \label{S_meth}

There are typically two alternative approaches to optimisation in this context \citep{Kor97}: (i) the martingale method \citep*[e.g.,][]{GuStZh21}, and (ii) the dynamic programming approach \citep[e.g.,][]{KrSt13}. We use the former, which involves the following steps:
\begin{arcenum}
\item Postulate an optimal terminal wealth design form;
\item Verify the existence of optimal designs for that form;
\item Verify that the candidate form is indeed optimal.
\end{arcenum}
We follow those steps, the last of which is generally performed via path-wise optimisation.

\subsection{Unconstrained problem} \label{sec.U}

We start out by considering the problem of maximizing utility without any constraints, i.e. we consider the problem
\begin{equation}\label{Problem.U}
    \max_{\pi\in\Pi}\mathbb{E}\left[-\frac{1}{2}(k-X_T^\pi)^2\right],
\end{equation}
with
\begin{equation}\label{def.dxpt}
    dX_t^\pi=(1-\pi_t)(bdt +\sigma dW_t),
    \quad X(0)=x<k.
\end{equation}

\subsubsection{Candidate}

We define the family of reinsurance designs implicitly given by the terminal wealth,
\begin{equation}\label{st.U}
    X_T^{\pi_\lambda}:=k-\lambda Z_T,
\end{equation}
characterised by the parameter $\lambda>0$. Denote also its notation $\Pi_U$, i.e.
\begin{equation}\label{PiU}
    \Pi_U = \{\pi_\lambda: \lambda > 0\}
\end{equation}
and we say a design $\pi$ is of class $\Pi_U$ if $\pi\in\Pi_U$.

\subsubsection{Existence}

We want to show the following:
\begin{theorem}\label{thm.existence.U}
There exists a $\lambda > 0$ such that 
$\mathbb{E}[Z_TX_T^{\pi_\lambda}]=x$.
\end{theorem}
\begin{proof}\label{pf.existence.U}
Define the value
\begin{equation}\label{VU}
   V_U(\lambda):=\mathbb{E}[Z_TX_T^{\pi_\lambda}].
\end{equation}
 It is easy to see that $V_U$ is continuous with respect to its arguments (e.g. by direct computation via \eqref{app.help}).
 We see immediately from the definition $X_T^{\pi_\lambda}$ that, since $Z_T>0$, we have that
\begin{equation}
    \parD{\lambda}V_U<0,\quad \lim_{\lambda\rightarrow\infty}V_U(\lambda)=-\infty,\quad V_U(0)=k>x.
\end{equation}
Since $-\infty< x\leq k$, we conclude that there exists a $\lambda^*_U>0$ such that $V_U({\lambda^*_U})=x$.
\end{proof}

\begin{remark}
    Note we deliberately suppressed the dependence of $V_U$ on $x$, as we see later that the chosen $\lambda_U^*$ is independent of $x$.
\end{remark}

This gives our candidate design
\begin{equation}
    X_T^{\pi^*}:=X_T^{\pi_{\lambda^*_U}}=k-{\lambda^*_U} Z_T.
\end{equation}

\subsubsection{Optimality}

This section proves the following:
\begin{theorem}
\label{thm.optimal.U}
The design $\pi_{\lambda_U^*}$ is optimal with respect to 
\eqref{Problem.U}.
\end{theorem}

\begin{proof}
Consider the optimisation problem of maximizing the function $f$ over $B$ defined by
\begin{equation}
    f(B;\omega):=-\frac{1}{2}(k-B)^2-\lambda^*_U[Z_T(\omega)B-x].
\end{equation}
with support $B\in(-\infty,k]$.
This is a quadratic function with an axis of symmetry at the point $k-\lambda^*_U Z_T(\omega)$. Since this is smaller than $k$, we immediately conclude that the maximum over $(-\infty,k]$ must be obtained at the point $k-\lambda^*_U Z_T(\omega)$. Therefore, we have that 
$$B^*(\omega):=k-\lambda^* Z_T(\omega)$$
maximises $f(\cdot,\omega)$ over $(-\infty,k]$. In other words, we have that
$$f(X_T^{\pi_{\lambda^*_U}}(\omega),\omega)\geq f(X_T^\pi(\omega),\omega)$$
for any $\pi\in\Pi$.
Now, taking expectation and recalling that $\lambda^*_U$ is chosen such that $\mathbb{E}[Z_TX_T^{\pi_{\lambda^*_U}}]=x$, we see that
\begin{equation}
    \mathbb{E}\left[-\frac{1}{2}(k-X_T^{\pi^*})^2\right]=\mathbb{E}\left[-\frac{1}{2}(k-X_T^{\pi_{\lambda^*_U}})^2\right]\geq \mathbb{E}\left[-\frac{1}{2}(k-X_T^\pi)^2\right],
\end{equation}
which is our desired result.
\end{proof}

\subsection{Strict constraint: $X_T^\pi \ge C$}

In this section, we consider the problem of maximizing utility under a strict solvency constraint, i.e. we consider the problem
\begin{equation}\label{Problem.C}
    \max_{\pi\in\Pi}\mathbb{E}[-\frac{1}{2}(k-X_T^\pi)^2],
\end{equation}
subject to
\begin{equation}\label{Constraint.C}
    X^\pi_T\geq C,
\end{equation}
with
\begin{equation}
    dX_t^\pi=(1-\pi_t)(bdt +\sigma dW_t),
    \quad X(0)=x\in(C,k).
\end{equation}

\subsubsection{Candidate}

We define the family of designs implicitly given by the terminal wealth,
\begin{equation}\label{st.C}
    X_T^{\pi_\lambda^C}:=\max(k-\lambda Z_T,C),
\end{equation}
characterised by the parameter $\lambda>0$. Denote also $\Pi_C$ the associated family of designs, i.e.
\begin{equation}
    \label{PiC}
    \Pi_C = \{\pi_\lambda^C: \lambda > 0\},
\end{equation}
and we say $\pi$ is of class $\Pi_C$ if $\pi\in \Pi_C$.

\begin{figure}
		\centering
		\includegraphics[width=.49\textwidth]{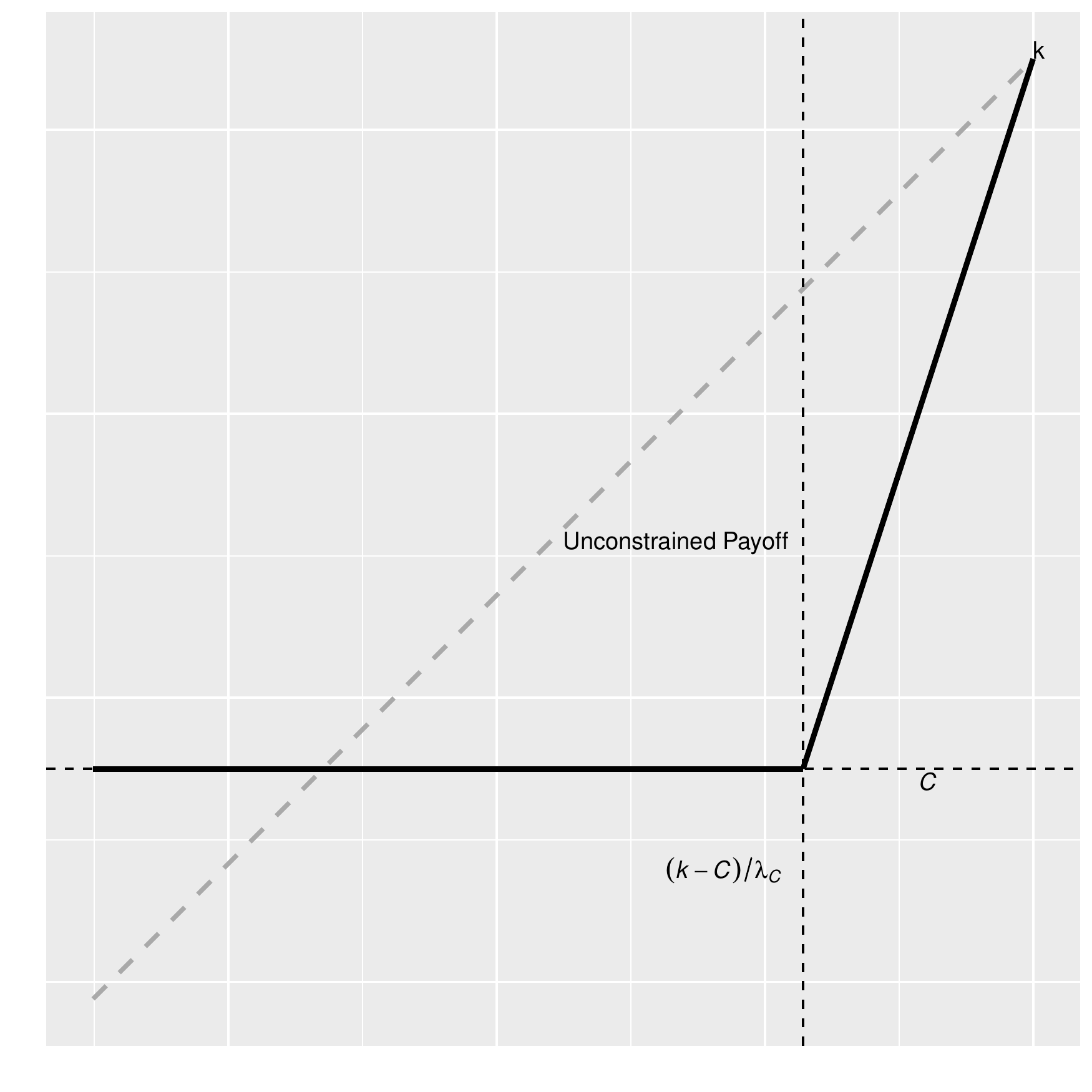}
		\caption{Strict constraint}
		\label{F_Strict} 
	\end{figure}

Figure \ref{F_Strict} displays the payoff of the unconstrained and strictly constrained cases as a function of the unconstrained payoff.  Obviously, with that construction, the unconstrained payoff with the dashed line forms a straight line with a slope equal to 1. The maximal payoff is $k$, so there the line ends. As a function of the unconstrained design, the constrained design appears as a solid line, capped off at the constraint level $c$ indicated by the dotted line and ending at $k$. The slope of the solid line is determined by the budget constraint or, said differently, such that the value of the two payoffs is the same, namely the initial capital. Figure \ref{F_Strict} also illustrates that if $C>k$, the unconstrained payoff also solves the constrained problem.

\begin{remark}\label{R_interpret}
The interpretation of Figures \ref{F_Strict}-\ref{F_optishapes} is performed analogously. The various constraints result in specific kink points, and the slopes are then determined such that the value of the option-like payoff equals the initial capital. In all these figures, the baseline is the optimal unconstrained design such that the kinks and the slopes illustrate how to exactly protect that reinsurance design with options to obey the constraints in the different optimization problems. 

Note that there are different ways to illustrate the payoff depending on the argument in the function. Here, we chose the unconstrained design as this gives an immediate interpretation of the impact of the various constraints.
\end{remark}

\subsubsection{Existence}
We want to show the following:
\begin{theorem}\label{thm.existence.C}
There exists a $\lambda > 0$ such that 
$\mathbb{E}[Z_TX_T^{\pi_\lambda^C}]=x$.
\end{theorem}
\begin{proof}

Define the value function
\begin{align*}
   V_C(\lambda):=\mathbb{E}[Z_TX_T^{\pi_\lambda^C}].
\end{align*}
 It is also easy to see that $V_C$ is continuous with respect to its arguments (e.g. by direct computation via \eqref{app.help}).

We see immediately from the definition $X_T^{\pi_\lambda^C}$ that, since $Z_T>0$, we have that
\begin{equation}
    \parD{\lambda}V_C(\lambda)<0,\quad \lim_{\lambda\rightarrow\infty}V_C(\lambda)=C<x,\quad V_C(0)=k>x.
\end{equation}
Therefore, we conclude that there exists a $\lambda^*_C>0$ such that $V_U(\lambda^*_C)=x$.
\end{proof}

This gives our candidate design
\begin{equation} \label{E_optiStrict}
    X_T^{\pi^*_C}:=X_T^{\pi_{\lambda^*_C}^C}=\max(k-{\lambda^*_C} Z_T,C).
\end{equation}

\subsubsection{Optimality}

\begin{theorem}
\label{thm.optimal.C}
The design $\pi_C^*$ is optimal with respect to 
\eqref{Problem.C}, subject to \eqref{Constraint.C}.
\end{theorem}

\begin{proof}
The proof is similar to that of Theorem \ref{thm.optimal.U}. See Appendix \ref{app.proof.C} for details.
\end{proof}

\subsection{Probability (VaR) constraint: $\mathbb{P}[X_T^\pi \ge C]\ge 1-\epsilon$} \label{sec.prob.constraint}

In this section, we consider the problem of maximizing the utility under a VaR-type of solvency constraint where the terminal wealth needs to exceed a certain level with a given tolerance probability. The objective combines a utility maximization perspective with a tolerance shortfall probability relative to some given level. The tolerance probability and the level from which we measure shortfall are ingredients in the objective. These can be set internally as part of the economic management of the surplus and/or externally as an ingredient in the solvency regulation. The problem is formalized by
\begin{equation}\label{Problem.P}
    \max_{\pi\in\Pi}\mathbb{E}\left[-\frac{1}{2}\left(k-X_T^\pi\right)^2\right],
\end{equation}
subject to
\begin{equation}\label{constraint.P}
    \mathbb{P}[X_T^\pi\geq C]\geq  1-\varepsilon,
\end{equation}
with
\begin{equation}
    dX_t^\pi=(1-\pi_t)(bdt +\sigma dW_t),
    \quad X(0)=x\in[C,k].
\end{equation}
Thus, the parameters $C$ and $\epsilon$ specify the shortfall and probability tolerance levels, respectively. 

\subsubsection{Candidate}

We define the family of designs implicitly given by the terminal wealth,
\begin{equation}\label{st.P}
    X_T^{\pi_{\lambda,c}}=
    \begin{cases}
    k-\lambda Z_T, ~&k-\lambda Z_T\notin[c,C],\\
    C , ~&k-\lambda Z_T\in[c,C],
    \end{cases}
\end{equation}
characterised by the constants $\lambda,c>0$ with $c\leq C\leq k$. We also denote the family $\Pi_P$, i.e.
\begin{equation}
    \label{PiP}
    \Pi_P:=\{\pi_{\lambda,c}: \lambda>0, c\leq C\}.
\end{equation}

Before moving on to the calculations, for convenience, we take the following transformation
\begin{align}
    &g_1 = \frac{k-C}{\lambda},\\
    &g_2 = \frac{k-c}{\lambda},
\end{align}
and we define the value function
\begin{equation}
    \label{VP}
    V_P(g_1,g_2):=\mathbb{E}[Z_TX_T^{\pi_{\lambda, c}}],
\end{equation}
where we have $0<g_1\leq g_2$. It is also easy to see that $V_P$ is continuous with respect to its arguments (e.g. by direct computation via \eqref{app.help}).

Figure \ref{F_VaR} is the VaR-constraint version of Figure \ref{F_Strict}. Again, the unconstrained payoff appears as the dashed line with slope 1. As is the case for the strict constraint, the VaR-constraint payoff is capped off at the level $C$, but only so until the unconstrained payoff is small enough. Then there is a jump in the payoff and from that point and down the payoff follows the straight line we had above $C$. The combination of (i) the point of discontinuity and (ii) the slope of the line above $C$ (identical to that below the point of discontinuity, too) are two unknowns determined by the two equations formed by the VaR-constraint and the budget constraint; see \eqref{E_VaRandbudgetconstraints} below. The figure also illustrates that if the VaR constraint is readily fulfilled by the unconstrained payoff (see \eqref{E_Unotopti} below), then the unconstrained payoff also solves the constrained problem.

	\begin{figure}[htb]
		\centering
		\includegraphics[width=.49\textwidth]{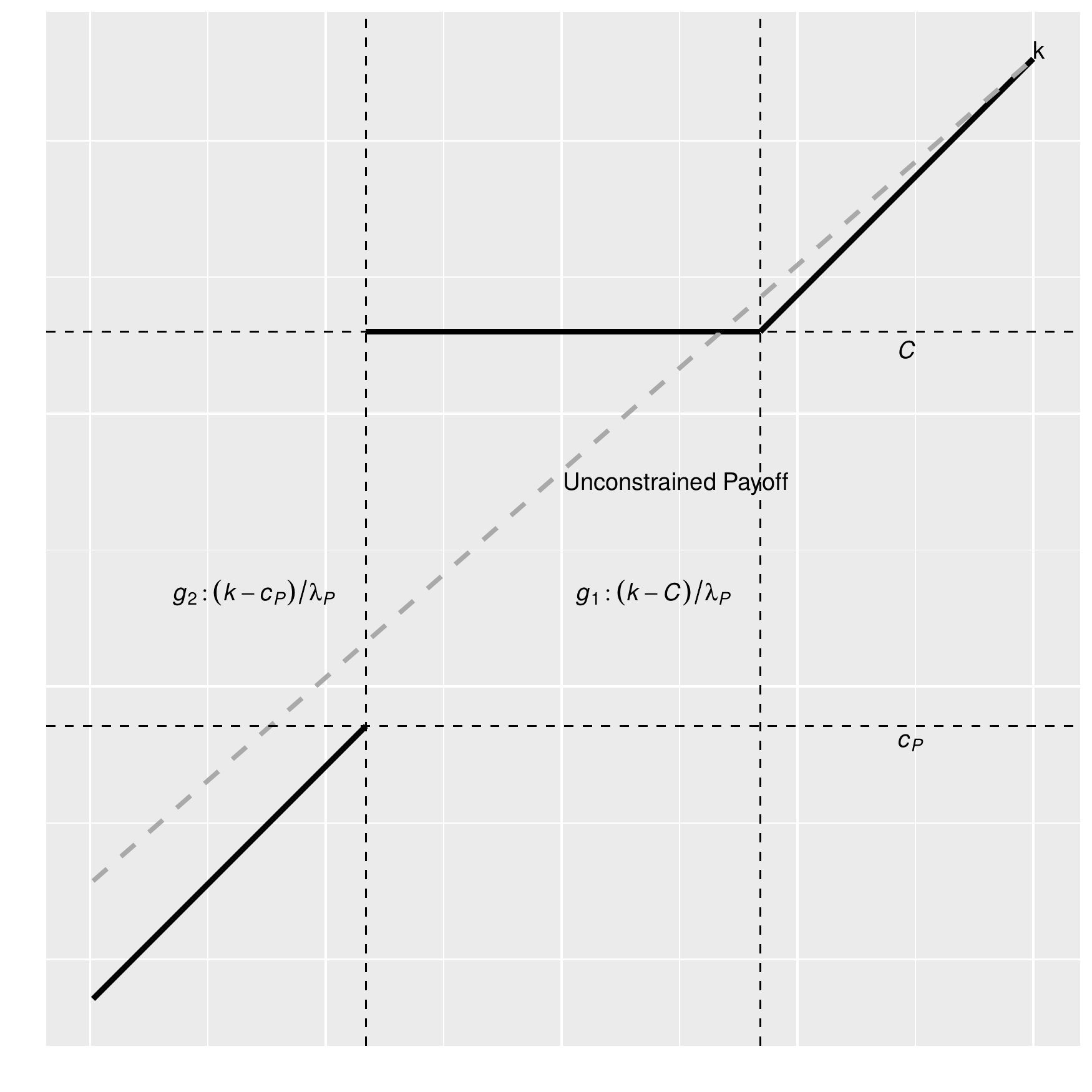}
		\caption{VaR constraint}
		\label{F_VaR}
	\end{figure}
	
\subsubsection{Existence}

We want to show the following:
\begin{theorem}\label{thm.existence.P}
Suppose $\mathbb{P}[X_T^{\pi_U^*}\geq C]<1-\varepsilon$. There exists $g_1$ and $g_2$ with $0<g_1\leq g_2$ such that 
$V_P(g_1,g_2)=x$.
\end{theorem}
\begin{proof}

The quantities $\mathbb{E}[Z_TX_T^{\pi_{\lambda,c}}]$ and $\mathbb{P}[Z_T\leq g_2]$ can be computed with ease (e.g. via \eqref{app.help}). 
Solving $\mathbb{P}[X^{\pi_{\lambda,c}}_T\geq C] = \mathbb{P}[Z_T\leq g_2] = 1-\varepsilon$, we get
\begin{equation}
    g_2^* = e^{-\frac{1}{2}\beta^2T-\beta\sqrt{T}\Phi^{-1}(1-\varepsilon)},
\end{equation}
where $\Phi$ is the distribution function of a standard Gaussian random variable.

We are now left with $g_1 = (k-C)/\lambda$.
First, consider when $g_1\uparrow g_2^*$. This implies that the value function $V_P(\cdot,g_2^*)$ converges to 
$V_U(g_2^*)$, the value function of a $\Pi_U$ design considered in Section \ref{sec.U}. Since by assumption $\pi_U^*$ is not optimal, i.e.
\begin{equation}\label{E_Unotopti}
    \mathbb{P}[X_T^{\pi_U^*}\geq C]<1-\varepsilon,
\end{equation}
then with $g_1\uparrow g_2^*$ and $\mathbb{P}[Z_T\leq g_2^*] = 1-\varepsilon$, we can deduce that 
$$
  \lambda < \lambda_U^* \implies k-\lambda e^{\beta^2T}> k-\lambda_U^* e^{\beta^2T}
  \iff 
  \lim_{g\uparrow g_2^*}V_P(g,g_2^*)>x
  .
$$

On the other hand, consider $g_1\downarrow 0$. This implies that the design $\lambda\uparrow\infty$ and $c\downarrow-\infty$. This is effectively a design of class $\pi_C$.

Note $\beta<0$, $\lambda g_1 = k-C$ and therefore invoking the value function we see that when $g_1\downarrow 0$, 
$$
    V_P(g_1,g_2^*)
    \rightarrow k-(k-C)\Phi\left(\frac{\frac{1}{2}\beta^2T-\log(g_2^*)}{\beta\sqrt{T}}\right)
        -\lambda e^{\beta^2T}\left(
            1- \Phi\left(\frac{\frac{3}{2}\beta^2T-\log(g_2^*)}{\beta\sqrt{T}}\right) 
        \right) \downarrow -\infty.
$$
Consequently, due to the continuity property $g_1\mapsto V_P(g_1,g_2^*)$, when $g_1$ increases from $0$ to $g_2^*$, there must exist a $g_1^*$ such that $\mathbb{E}[Z_T X^{\pi_{\lambda,c}}_T] =x$.
This gives the desirable pair $(g_1^*, g_2^*)$. 
\end{proof}

In other words, we have constructed a design of the form \eqref{st.P} such that
\begin{equation} \label{E_VaRandbudgetconstraints}
    \begin{cases}
        &\mathbb{P}\left[X^{\pi_{\lambda_P^*,c_P^*}}_T\geq C\right]=1-\varepsilon,\\
        &\mathbb{E}\left[Z_T X^{\pi_{\lambda_P^*,c_P^*}}_T\right]=x,
    \end{cases}
\end{equation}
provided that $\pi_U^*$ is not optimal.

\subsubsection{Optimality}

\begin{theorem}
\label{thm.optimal.P}
The design $\pi_{\lambda_P^*,c_P^*}$ is optimal with respect to 
\eqref{Problem.P}, subject to \eqref{constraint.P}.
\end{theorem}

\begin{proof}
The proof is similar to that of Theorem \ref{thm.optimal.U}. See Appendix \ref{app.proof.P} for details.
\end{proof}

\subsection{Expected shortfall constraint under $\mathbb{P}$: $\mathbb{E}[(C-X_T^\pi)_+ ]\le \nu$}

In this section, we consider the problem of maximizing utility under an Expected Shortfall-type of solvency constraint where the expected shortfall of the terminal wealth relative to some given level must exceed a given tolerance expected shortfall. As in the previous section, the objective combines a utility maximization perspective with a risk measure tolerance. Still, the tolerance is formulated with respect to the expected shortfall probability, where the shortfall is measured relative to some given level. The tolerance expected shortfall and the level from which shortfall is measured are ingredients in the objective. These can be set internally as part of the economic management of the surplus and/or externally as ingredients in a solvency regulation. The problem is formalised by

\begin{equation}\label{Problem.S}
    \max_{\pi\in\Pi}\mathbb{E}\left[-\frac{1}{2}\left(k-X_T^\pi\right)^2\right],
\end{equation}
subject to
\begin{equation}\label{Constraint.S}
    \mathbb{E}\left[\left(C-X_T^\pi\right)_+\right]\leq \nu,
\end{equation}
with
\begin{equation}
    dX_t^\pi=(1-\pi_t)(bdt +\sigma dW_t),
    \quad X(0)=x\in[C,k].
\end{equation}

Thus, the parameters $C$ and $\nu$ specify the shortfall and expected shortfall tolerance levels, respectively.

\subsubsection{Candidate}

We define the family of designs implicitly given by the terminal wealth,
\begin{equation}
    X_T^{\pi_{\lambda,\gamma}}=
    \begin{cases}
    k-\lambda Z_T, ~&k-\lambda Z_T\geq C\\
    C , ~&k-\lambda Z_T\in[C-\gamma,C]\\
    k-\lambda Z_T+\gamma,~&k-\lambda Z_T<C-\gamma
    \end{cases},
\end{equation}
characterised by the constants $\lambda>0$ and $\gamma\geq 0$, with $C\leq k$. Denote $\Pi_E$ the family of designs, i.e.
\begin{equation}
    \label{st.E}
    \Pi_E = \{\pi_{\lambda,\gamma}: \lambda>0, \gamma\geq 0\}.
\end{equation}
Similar to Section \ref{sec.prob.constraint}, we make the following transformations
$$
\begin{cases}
    h_1 =~& \frac{k-C}{\lambda},\\
    h_2 =~& \frac{k+\gamma-C}{\lambda},
\end{cases}
$$
and denote the value function 
\begin{equation}
    \label{VE}
    V_E(h_1,h_2) = \mathbb{E}\left[Z_TX_T^{\pi_{\lambda,\gamma}}\right].
\end{equation}
 It is also easy to see that $V_E$ is continuous with respect to its arguments (e.g. by direct computation via \eqref{app.help}).

	\begin{figure}[htb]
		\centering
		\includegraphics[width=.49\textwidth]{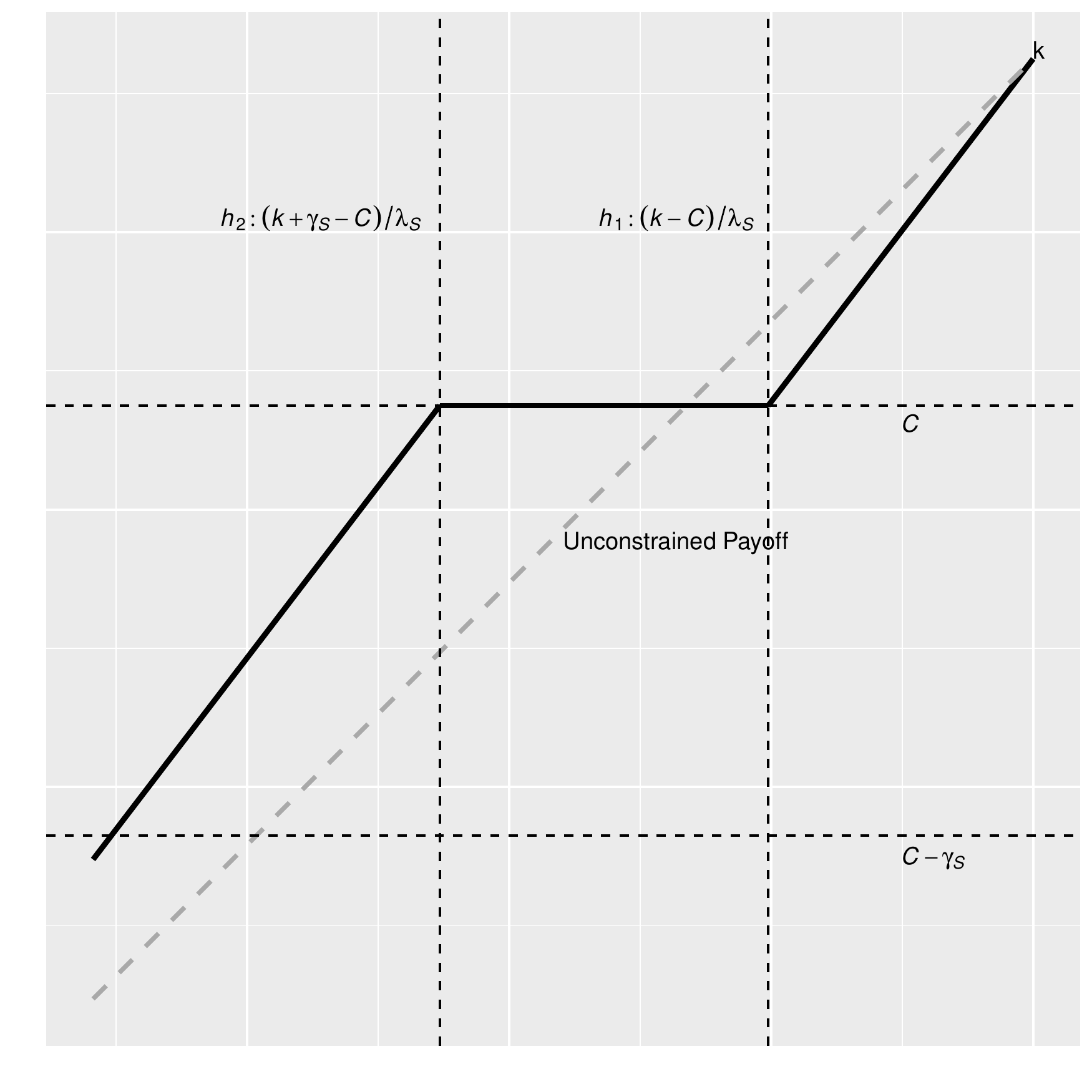}
		\caption{ES constraint under $\mathbb{P}$}
		\label{F_ES_P} 
	\end{figure}

Figure \ref{F_ES_P} is the $\mathbb{P}$-ES-constraint version of Figures \ref{F_Strict} and \ref{F_VaR}. The dashed line is the same. The upper part of the solid line is also the same. But instead of a point of discontinuity, the payoff has another kink from which it follows the same slope as in the upper part. Again, (i) the slope of the payoff (outside the constant part in the middle) and (ii) the point at which the line kinks in its lower end are two unknowns determined by the two equations formed by the shortfall and the budget constraint as in \eqref{E_ESPandbudgetconstraints} below.  The figure also illustrates that if the ES-constraint is readily fulfilled by the unconstrained payoff (that is, if \eqref{Constraint.S} is not verified under $\pi_U^*$), then the unconstrained payoff also solves the $\mathbb{P}$-ES-constraint problem.

\subsubsection{Existence}

We want to show the following:
\begin{theorem}\label{thm.existence.E}
Suppose $\mathbb{E}\left[\left(C-X_T^{\pi_U^*}\right)_+\right]>\nu$. There exists $h_1$ and $h_2$ with $h_1\leq h_2$ such that 
$V_E(h_1,h_2)=x$.
\end{theorem}

\begin{proof}

The expected shortfall is calculated as 
\begin{align}
    \mathbb{E}\left[\left(C-X_T^{\pi_{\lambda,\gamma}}\right)_+\right] =~& (C-\gamma-k)\mathbb{P}[Z_T>h_2]+\lambda \mathbb{E}[Z_T;Z_T>h_2]\nonumber\\
    =~& \lambda(\mathbb{E}[Z_T;Z_T>h_2] - h_2\mathbb{P}[Z_T>h_2]).
\end{align}
Equating this to the target $\nu$, we have
\begin{equation}\label{h2.lambda}
    \lambda = \frac{\nu}{\mathbb{E}[Z_T-h_2;Z_T>h_2]},
\end{equation}
a function of $h_2$. 

Note further that the function
\begin{equation}
    h\mapsto \mathbb{E}[Z_T-h;Z_T>h]
\end{equation}
is a function decreasing from positive to $0$ when $h$ increases from $0$ to $\infty$. Hence, $\lambda$ is guaranteed to be positive.

Using the same profile argument as in Section \ref{sec.prob.constraint} with the scheme
$$
    h_2\rightarrow \lambda \rightarrow h_1
$$
(where we have to make sure $h_1\leq h_2$), we consider the budget function $$h_2\mapsto V_E(h_1(\lambda(h_2)), h_2)-x,$$ and we show that it has a finite zero.

To start, consider $h_2\uparrow\infty$. This can be done by setting $\gamma = \lambda h_2 + C-k$, with $\lambda = \lambda(h_2)$ given by \eqref{h2.lambda}. Consequently, we have $\lambda\uparrow\infty$ and $h_1\downarrow 0$, which implies that $V_E(h_1,h_2)$ converges to
$$
    k- (k-C)-\lambda e^{\beta^2T} = C-\lambda e^{\beta^2T} < C \leq x,
$$
after evaluating $V_E$ (e.g. applying \eqref{app.help} on \eqref{VE} with definition \eqref{st.E}) and taking limit.

Next, consider the case where $h_2$ is decreasing from some large constant $K$ to $0$, where at $K$ we have that $h_2>h_1(\lambda(h_2))$.
Ignoring the definition of $h_2$ at the moment, when $h_2\downarrow 0$, $\lambda(h)\rightarrow\nu$ (see \eqref{h2.lambda}) and hence $h_1(\lambda(h_2))\rightarrow (k-C)/\nu>0$, which implies $h_1(\lambda(h))>h$ for small enough $h$. Therefore, by the continuity of $h_1\circ \lambda$, there is a $h_0\in(0,\infty)$ such that $h_0 = h_1(\lambda(h_0))$ and the corresponding $\gamma$ value is $0$. Denote $\lambda_3:=\lambda(h_0)$, we have that $\pi_{\lambda,\gamma} = \pi_{\lambda_3} \in \Pi_U$. It remains to show that
$$
    \lambda_3<\lambda_U^*,
$$
as this implies that
$$
    x = V_U(\lambda^*_U) < V_U(\lambda_3) = 
    V_E(h_1(\lambda(h_0)), h_0),
$$
i.e. there is a $h_2\in (h_0,K)$ such that $V_E(h_1(\lambda(h_2)), h_2)=x$.

By direct computation, for the design $\pi_{\lambda}\in\Pi_U$, its expected shortfall is
$$
    (C-k)\mathbb{P}\left[Z_T>\frac{k-C}{\lambda}\right] + \lambda \mathbb{E}\left[Z_T;Z_T>\frac{k-C}{\lambda}\right] = 
    \lambda \mathbb{E}\left[Z_T - \frac{k-C}{\lambda};Z_T>\frac{k-C}{\lambda}\right];
$$
an increasing function of $\lambda$. From the assumption that $\pi^*_U$ is not optimal, we have that
$$
    \mathbb{E}\left[\left(C-X^{\pi_{\lambda^*_U}}_T\right)_+\right]> \nu = \mathbb{E}\left[\left(C-X^{\pi_{\lambda_3}}_T\right)_+\right]
    \iff \lambda^*_U>\lambda_3
    ,
$$
which completes the proof.
\end{proof}

In other words, supposing that $\pi_U^*$ is not optimal, we have constructed a design \begin{equation}
    \pi^*_S:=\pi_{\lambda^*_S,\gamma^*_S}, 
\end{equation}
which generates a terminal wealth
\begin{equation} \label{E_optiP}
    X_T^{\pi^*_S}=
    \begin{cases}
    k-\lambda^*_S Z_T, ~&k-\lambda^*_S Z_T\geq C,\\
    C , ~&k-\lambda^*_S Z_T\in[C-\gamma^*_S,C],\\
    k-\lambda^*_S Z_T+\gamma^*_S,~&k-\lambda^*_S Z_T<C-\gamma^*_S,
    \end{cases}
\end{equation}
with
\begin{equation} \label{E_ESPandbudgetconstraints}
    \begin{cases}
        &\mathbb{E}\left[\left(C-X_T^{\pi^*_S}\right)_+\right]=\nu,\\
        &\mathbb{E}\left[Z_T X_T^{\pi^*_S}\right]=x.
    \end{cases}
\end{equation}

\subsubsection{Optimality}

\begin{theorem}
\label{thm.optimal.E}
The design $\pi_{\lambda^*_S,\gamma^*_S}$ is optimal with respect to 
\eqref{Problem.S}, subject to \eqref{Constraint.S}.
\end{theorem}

\begin{proof}
The proof is similar to that of Theorem \ref{thm.optimal.U}. See Appendix \ref{app.proof.E} for details.
\end{proof}

\subsection{Expected shortfall constraint under $\mathbb{Q}$: $\mathbb{E}[Z_T(C-X_T^\pi)_+ ]\le \nu$} \label{S_ESQ}

In this section, we consider the problem of maximizing utility under a $\mathbb{Q}$-Expected Shortfall-type of solvency constraint where the expected shortfall of the terminal wealth relative to some given level must exceed a given tolerance $\mathbb{Q}$-expected shortfall. As in the previous section, the objective combines a utility maximization perspective with a risk measure tolerance, which is with respect to the $\mathbb{Q}$-expected shortfall probability where the shortfall is measured relative to some given level. 

The tolerance $\mathbb{Q}$-expected shortfall and the level from which shortfall is measured are ingredients in the objective. These are set internally as part of the economic management of the surplus. In contrast to the risk constraints in the previous sections, it is not realistic to have a solvency constraint formulated in terms of $\mathbb{Q}$-expectations. However, it makes perfect sense from an internal risk management point of view. Namely, the risk tolerance level corresponds to the value it costs to eliminate the risk. Conservative risk management means that the value of the retained loss is small. In particular, when communicating risk management processes to financial specialists, it may work well to communicate a risky position in terms of how much it costs to cover it. 

The problem is formalised by
\begin{equation}\label{Problem.Q}
    \max_{\pi\in\Pi}\mathbb{E}\left[-\frac{1}{2}\left(k-X_T^\pi\right)^2\right],
\end{equation}
subject to
\begin{equation}\label{Constraint.Q}
    \mathbb{E}\left[Z_T\left(C-X_T^\pi\right)_+\right]\leq \nu.
\end{equation}
with
\begin{equation}
    dX_t^\pi=(1-\pi_t)(bdt +\sigma dW_t),
    \quad X(0)=x\in[C,k].
\end{equation}
Thus, the parameters $C$ and $c$ specify the shortfall level and the $\mathbb{Q}$-Expected Shortfall tolerance level, respectively. The factor $Z_T$ in \eqref{Constraint.Q} formalizes that the expectation of the loss $(C-X_T^\pi)_+$ is taken under the risk-neutral measure $\mathbb{Q}$.

\subsubsection{Candidate}

We define the family of designs implicitly given by the terminal wealth
\begin{equation}
    X_T^{\pi_{\lambda,\delta}}=\begin{cases}
    k-{\lambda} Z_T(\omega),\quad &\text{if }Z_T(\omega)\leq \frac{k-C}{\lambda},\\
    C,\quad &\text{if }Z_T(\omega)\in\left[\frac{k-C}{\lambda},\frac{k-C}{\delta}\right],\\
    k-\delta Z_T(\omega),\quad &\text{if }Z_T(\omega)>\frac{k-C}{\delta},
\end{cases}
\end{equation}
characterised by the constants $\lambda>0$ and $\delta\in(0,\lambda]$, with $C\leq k$. Denote also $\Pi_\mathbb{Q}$ the family of designs, i.e.
\begin{equation}
    \label{PiQ}
    \Pi_\mathbb{Q} = \{ \pi_{\lambda,\delta}: 0 < \delta \leq \lambda \},
\end{equation}
and the value function
\begin{equation}
    \label{VQ}
    V_\mathbb{Q}(\lambda, \delta) = \mathbb{E}\left[Z_TX^{\pi_{\lambda, \delta}}\right].
\end{equation}
 It is also easy to see that $V_\mathbb{Q}$ is continuous with respect to its arguments (e.g. by direct computation via \eqref{app.help}).

	\begin{figure}
		\centering
		\includegraphics[width=.49\textwidth]{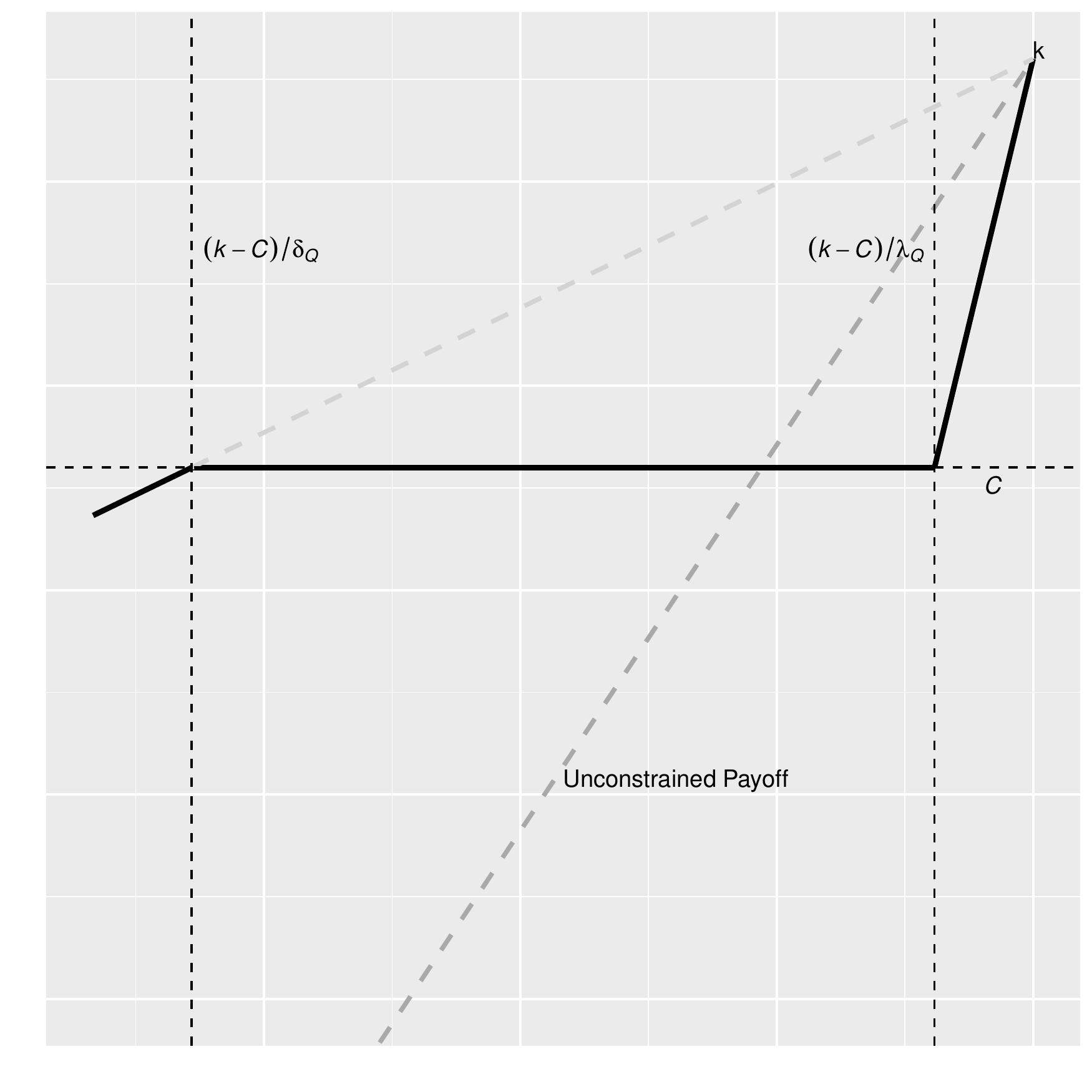}
		\caption{ES constraint under $\mathbb{Q}$}
		\label{F_ES_Q}
	\end{figure}
	
Figure \ref{F_ES_Q} is the $\mathbb{Q}$-ES-constraint version of Figures \ref{F_Strict}, \ref{F_VaR}, and \ref{F_ES_P}. Now, compared to Figure \ref{F_ES_P}, the slope from the lower kink of the solid line and downwards is different from the slope from the upper kink and upwards. It may look like we have a further degree of freedom (one more unknowns than equations) because the lower slope is not equal to the upper slope. However, the lower slope is now fixed by the point of the kink as the lower part of the payoff forms a straight line through $k$ (a light dashed grey line illustrates this). As before, the essential elements of the figure are determined by the constraints \eqref{E_ESQandbudgetconstraints}. The figure also illustrates that if the unconstrained payoff readily fulfils the ES-constraint (if \eqref{Constraint.Q} is not verified under $\pi_U^*$), then the unconstrained payoff also solves the $\mathbb{Q}$-ES-constraint problem.

\subsubsection{Existence}

We want to show the following:
\begin{theorem}\label{thm.existence.Q}
Suppose $\mathbb{E}\left[Z_T\left(C-X_T^{\pi_U^*}\right)_+\right]>\nu$. There exists $\lambda$ and $\delta$ with $0<\delta\leq \lambda$ such that 
$V_\mathbb{Q}(\lambda,\delta)=x$.
\end{theorem}

\begin{proof}

From the definition of the design, we have that
\begin{equation}\label{shortfall.Q.equiv}
    \mathbb{E}\left[Z_T\left(C-X_T^{\pi_{\lambda, \delta}}\right)_+\right] = \mathbb{E}\left[Z_T\left(C-X_T^{\pi_\delta}\right)_+\right]
\end{equation}
and that such quantity as a function of $\delta$ is increasing.

Next, we show that there is a unique $\delta^*_\mathbb{Q}$ such that \eqref{shortfall.Q.equiv} equals $\nu$. When $\delta\downarrow 0$, it is easily seen from above that the shortfall quantity is $0$. In other words, we have $\lim_{\delta\downarrow 0}V_\mathbb{Q}(\lambda, \delta)\rightarrow V_C(\lambda)$.

On the other hand, it is clear that the value of $\lambda$ does not affect the shortfall, provided that it is greater than the chosen $\delta$ value. Hence, we use a 
$\lambda \geq \lambda^*_U$. When $\delta\uparrow\lambda$, the value function converges to that of the design $\pi_{\lambda}$. With the assumption $\lambda \geq \lambda_U^*$ and that $\pi_{\lambda_U^*}$ is not optimal, we have
$$
    \mathbb{E}\left[Z_T\left(C-X_T^{\pi_\lambda}\right)_+\right] \geq \mathbb{E}\left[Z_T\left(C-X_T^{\pi_{\lambda^*_U}}\right)_+\right] > \nu,
$$
which shows the existence of $\delta^*_\mathbb{Q}\in(0,\lambda^*_U)$.

Thus, it remains to determine $\lambda$. Since $\delta^*_\mathbb{Q}<\lambda^*_U$, we have
$$
    V_\mathbb{Q}(\delta_\mathbb{Q}^*,\delta_\mathbb{Q}^*) = V_U(\delta_\mathbb{Q}^*) >
    V_U(\lambda_U^*) = x.
$$
On the other hand, notice that
$$
    \lambda \mathbb{E}\left[Z_T^2;Z_T\leq \frac{k-C}{\lambda}\right] =
        \lambda e^{\beta^2T}\Phi\left(\frac{\frac{3}{2}\beta^2T - log(\frac{k-C}{\lambda})}{\beta\sqrt{T}}\right)\\
    = e^{\beta^2T} \lambda\Phi\left(\frac{\frac{3}{2}\beta^2T-\log(k-C)}{\beta\sqrt{T}}+\frac{log(\lambda)}{\beta\sqrt{T}}\right)
$$
is of the form $A\lambda\Phi(a\log(\lambda)+b)$ (with $a<0$), which goes to $0$ when $\lambda\uparrow\infty$.
Recall also the definition of $\delta_\mathbb{Q}^*$:
$$
   \nu = \mathbb{E}\left[Z_T\left(C-\left(k-\delta_\mathbb{Q}^* Z_T\right)\right);Z_T>\frac{k-C}{\delta_\mathbb{Q}^*}\right]
$$
Therefore, when $\lambda\uparrow\infty$, we have
\begin{align*}
    V_\mathbb{Q}(\lambda, \delta^*_\mathbb{Q})
    \rightarrow~& 
        \mathbb{E}\left[Z_T(C);Z_T\leq \frac{k-C}{\delta^*_\mathbb{Q}}\right]
        + \mathbb{E}\left[Z_T\left(k-\delta^*_\mathbb{Q} Z_T\right);Z_T > \frac{k-C}{\delta^*_\mathbb{Q}}\right] \\
    =~&    \mathbb{E}\left[Z_T(C)\right]
        + \mathbb{E}\left[Z_T\left(k-\delta^*_\mathbb{Q} Z_T-C\right);Z_T > \frac{k-C}{\delta^*_\mathbb{Q}}\right] \\
    =~& C-\nu < x.
\end{align*}
Along with the usual continuity argument, existence is established.
\end{proof}

In other words, supposing $\pi_U^*$ is not optimal, we have constructed a design \begin{equation}
    \pi^*_\mathbb{Q}:=\pi_{\lambda^*_\mathbb{Q},\delta^*_\mathbb{Q}}, 
\end{equation}
which generates a terminal wealth
\begin{equation} \label{E_optiQ}
    X_T^{\pi^*_\mathbb{Q}}=\begin{cases}
    k-{\lambda^*_\mathbb{Q}} Z_T(\omega),\quad &\text{if }Z_T(\omega)\leq \frac{k-C}{\lambda^*_\mathbb{Q}},\\
    C,\quad &\text{if }Z_T(\omega)\in\left[\frac{k-C}{\lambda^*_\mathbb{Q}},\frac{k-C}{\delta^*_\mathbb{Q}}\right],\\
    k-\delta^*_\mathbb{Q} Z_T(\omega),\quad &\text{if }Z_T(\omega)>\frac{k-C}{\delta^*_\mathbb{Q}},
\end{cases}
\end{equation}
with
\begin{equation} \label{E_ESQandbudgetconstraints}
    \begin{cases}
        &\mathbb{E}\left[Z_T\left(C-X_T^{\pi^*_\mathbb{Q}}\right)_+\right]=\nu,\\
        &\mathbb{E}\left[Z_T X_T^{\pi^*_\mathbb{Q}}\right]=x.
    \end{cases}
\end{equation}

\subsubsection{Optimality}

\begin{theorem}
\label{thm.optimal.Q}
The design $\pi_{\lambda^*_\mathbb{Q},\delta^*_\mathbb{Q}}$ is optimal with respect to 
\eqref{Problem.Q}, subject to \eqref{Constraint.Q}.
\end{theorem}
\begin{proof}
The proof is similar to that of Theorem \ref{thm.optimal.U}. See Appendix \ref{app.proof.Q} for details.
\end{proof}

\section{Numerical illustrations} \label{S_NI}

In this section, we illustrate the path of the controlled $X_t^{\pi^*}$ for the different constraints considered in this paper, as well as its associated optimal $\pi_t^*$. 

The \verb_R_ code for computing the parameters in the optimal designs based on the basic parameters, payoff graphs, and paths for given seed are all available on GitHub at the following address: \url{https://github.com/agi-lab/LQ-reinsurance-under-constraints}.

\subsection{Parameters}

The initial hyperparameters are 
\begin{equation}
a=0.2, \quad b=0.5, \quad \sigma=1.2, \;\text{ and }x = 2
\end{equation}
for the surplus dynamics, and 
\begin{equation}
\tilde{k}= 5, \quad \tilde{C} = 0, \quad \epsilon =0.01, \;\text{ and }\nu = 0.1
\end{equation}
for the solvency constraints, with timeframe $T=5$.

This leads to the following parameters in the optimal designs:
\begin{eqnarray}\label{E_optipara}
\lambda_U^* = 1.888951 \text{ (unconstrained)}  \label{E_NI_para1} \\
\lambda_C^* = 5.828629 \text{ (strict constraint)} \label{E_NI_para2} \\
\left\{\begin{array}{c}
\lambda_P^* = 2.159931  \\
c^* = -5.725147 
\end{array}\right.
\text{ (probability (VaR) constraint)} \label{E_NI_para3} \\
\left\{\begin{array}{c}
\lambda_S^* = 2.472898 \\
\gamma_S^* = 6.201261 
\end{array}\right.
\text{ (expected shortfall under }\mathbb{P}\text{)} \label{E_NI_para4}
\\
\left\{\begin{array}{c}
\lambda_\mathbb{Q}^* = 5.199066  \\
\delta_\mathbb{Q}^* = 0.6094314 
\end{array}\right.
\text{ (expected shortfall under }\mathbb{\mathbb{Q}}\text{)} \label{E_NI_para5}
\end{eqnarray}

These parameters are the ones that were used to produce Figures \ref{F_Strict}--\ref{F_ES_Q}. The shapes of the corresponding optimal terminal payoffs are further compared in the following subsection. 

\subsection{Comparison of the optimal terminal payoffs} \label{S_terminalpayoff}

Comparison of the Strict constraint payoff \eqref{st.C} and $\mathbb{Q}$-ES constraint payoff \eqref{E_optiQ} with their optimal parameters \eqref{E_NI_para2} and \eqref{E_NI_para5} respectively, suggests that those are quite similar. Analogously, comparison of Unconstrained payoff \eqref{st.U}, VaR constraint payoff \eqref{st.P} and $\mathbb{P}$-ES constraint payoff \eqref{E_optiP} with their respective optimal parameters \eqref{E_NI_para1}, \eqref{E_NI_para3}, and \eqref{E_NI_para4} also shows a similar shape, especially for lower values of $Z_T$. These similarities are illustrated in Figure \ref{F_optishapes}.

\begin{figure}[!htb]
	\centering	
	
	\begin{subfigure}{0.49\textwidth}
		\centering
		\includegraphics[width=\textwidth]{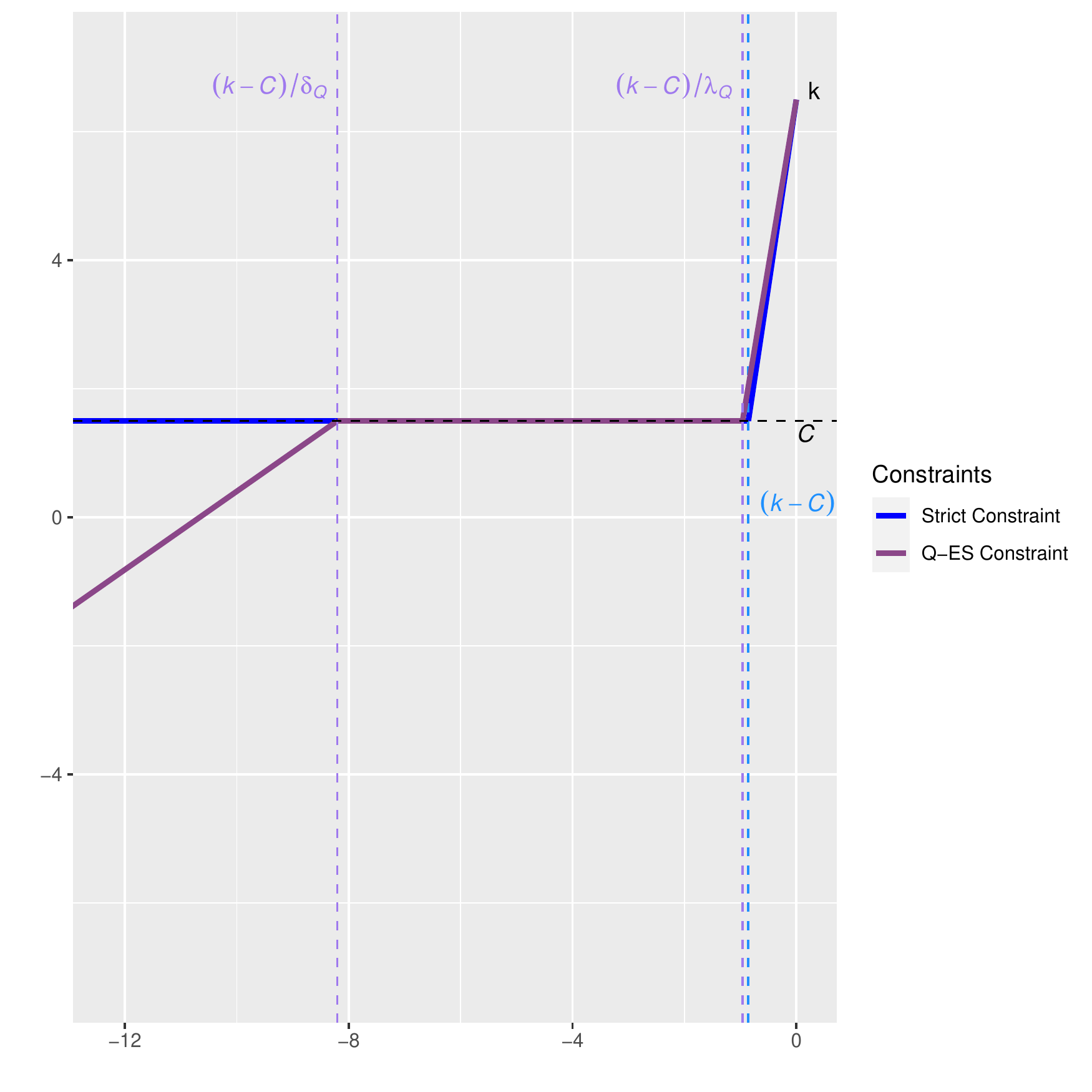}
		\caption{Strict constraint with ES constraint under $\mathbb{Q}$}
		\label{F_NI_Strict} 
	\end{subfigure}
	\begin{subfigure}{0.49\textwidth}
		\centering
		\includegraphics[width=\textwidth]{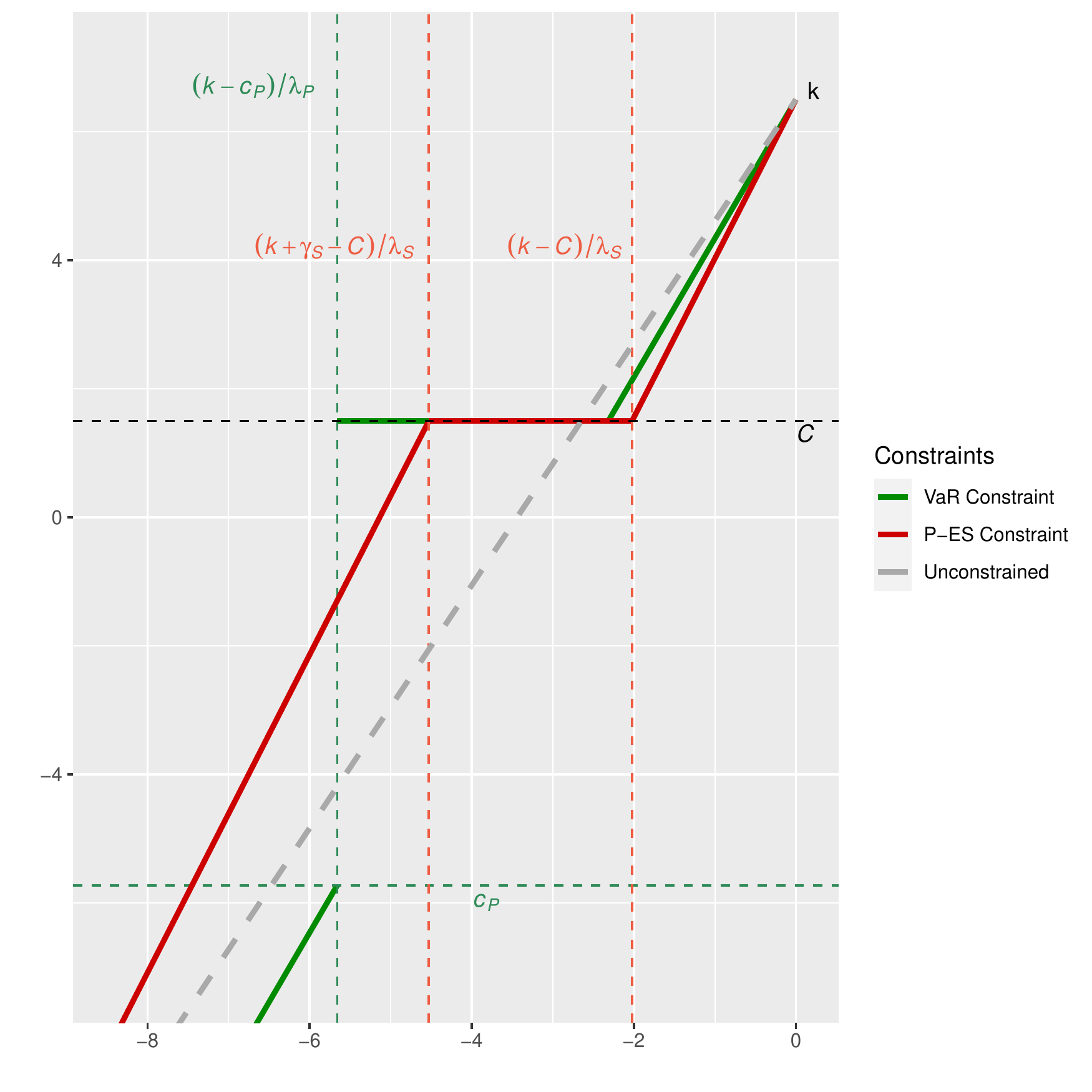}
		\caption{VaR constraint with ES constraint under $\mathbb{P}$}
		\label{F_NI_VaR}
	\end{subfigure}
	\caption{Comparison of the optimal terminal wealth shapes under optimal parameters \eqref{E_optipara} and different solvency constraints}
	\label{F_optishapes}
\end{figure}

Figures \ref{F_NI_Strict} and \ref{F_NI_VaR} show the payoff in the numerical example and are, thus, the numerical example version of the stylized payoff illustrated in Figures \ref{F_Strict}-\ref{F_ES_Q}. Panel (a) shows that the $\mathbb{Q}$-ES  constraint design and the strict constraint design have similar payoffs unless the realized unconstrained design is really low. Only then does the tolerance for a loss in the $\mathbb{Q}$-ES constraint design kick in. Conversely, Panel (b) shows that the $\mathbb{P}$-ES constraint design and the VaR constraint design have similar payoffs, again at least if the realised unconstrained design is not too low.

\begin{figure}[htb]
\centering
\includegraphics[width=.7\textwidth]{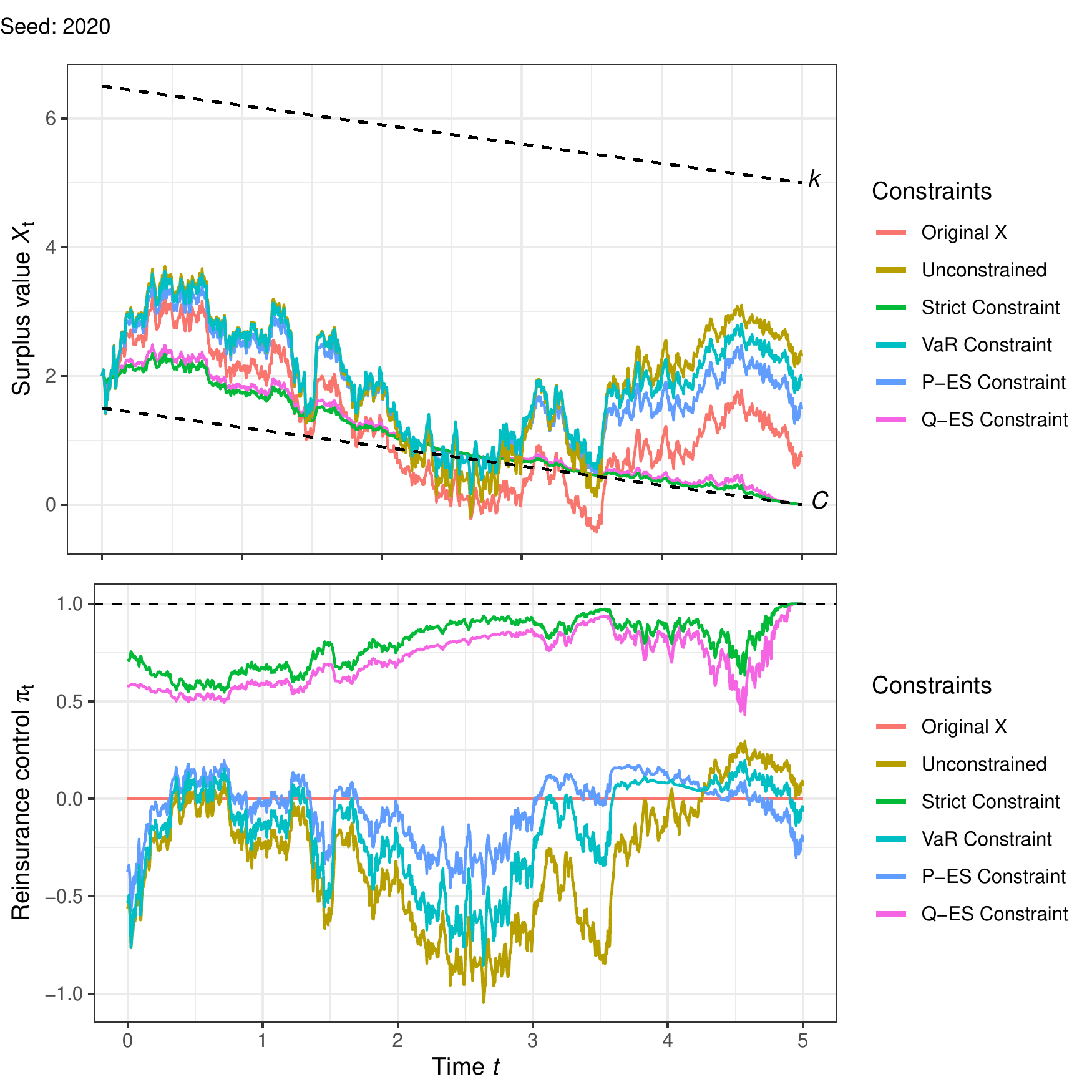}
\caption{Traces of the controlled $X_t^{\pi^*}$ and associated $\pi_t^*$ (seed 2020)}
\label{F_paths_2020}
\end{figure}

\subsection{Traces of the controlled $X_t^{\pi^*}$ and the associated $\pi_t^*$}

Here, the controlled path $X_t^{\pi^*}$ and associated reinsurance design $\pi_t^*$ are simulated by discretising $[0,T]$ over 1000 intervals, using different seeds. We compare the uncontrolled surplus $X_t$ (salmon line), with its controlled version under all five constraints considered in this paper. The paths after application of the optimal design are calculated according to Appendices \ref{A_paths} and \ref{ss.comp.pit}, respectively.

Figure \ref{F_paths_2020} considers one set of traces. First, note the uninsured surplus process (salmon). This is the surplus if the insurer does not reinsure to optimize performance \emph{and} there are no constraints on the surplus process. The surplus of an insurer who optimizes performance but has no constraints (gold) performs relatively well and beats the performance of all processes. For such a trace with relatively good performance, the constraints are burdens that penalise performance. When the trace is unfavourable, the constraints (and associated protection) help outperform the unconstrained design.

After the unconstrained surplus follows, in terms of performance, the VaR constraint (turquoise) and the $\mathbb{P}$-ES constraint processes (blue). These two constraints have relatively little effect on the final outcome. Also, they all perform better than the uninsured and uncontrolled surplus, but they are not radically different. Examination of the corresponding optimal reinsurance designs (turquoise, blue, and gold) confirms those comments: they are quite similar and circulate around 0.

The $\mathbb{Q}$-ES constraint (pink) and the Strict Constraint (green) traces perform alike and relatively badly. We see in the reinsurance designs that for these two realisations much of the business is sold off: reinsurance levels are above 0.5 and tending to 1. These designs  have little reward in this particular realisation because the surplus increases towards the end of the considered time frame. The surplus ends up for both much lower than the other constraints and ends up close to the tolerance level $C$. In more favourable endings (see, e.g., seed 1994), the reinsurance proportion becomes negative to ``regain'' some profits.

\begin{figure}[htb]
\centering
\includegraphics[width=.7\textwidth]{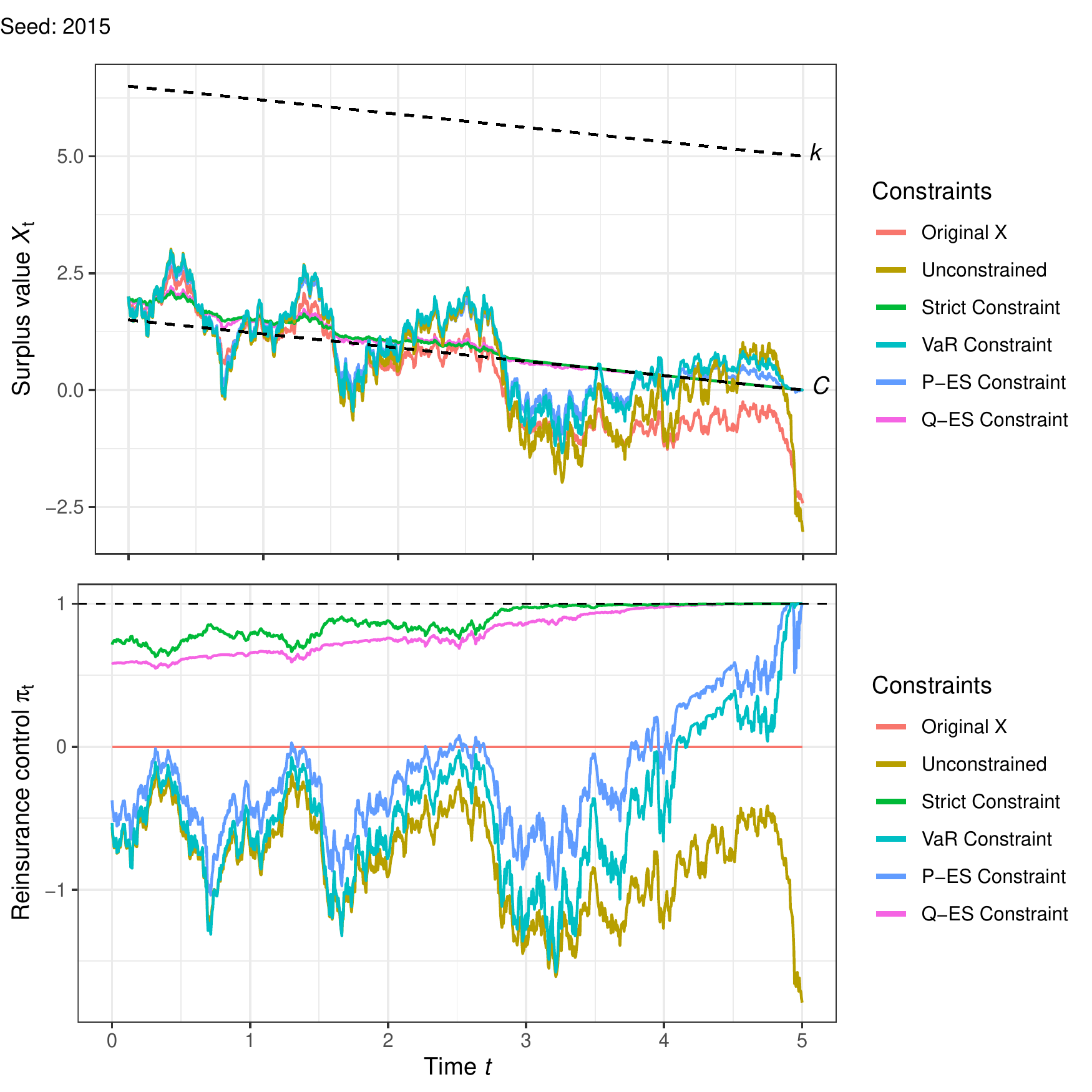}
\caption{Traces of the controlled $X_t^{\pi^*}$ and associated $\pi_t^*$ (seed 2015)}
\label{F_paths_2015}
\end{figure}

The clustering of the realised surplus and design processes discussed above was expected from the observations outlined in the previous subsection \ref{S_terminalpayoff}. Choosing a seed of 2 demonstrates this to an extreme level (this is not shown here, but could be generated from the code provided on GitHub). The clustering occurs and is illustrated when the unconstrained surplus performs relatively well, and downward protection is essentially a waste of money. We can expect less clustering for surpluses that perform poorly. 

Figure \ref{F_paths_2015} for seed 2015 illustrates a surplus path that finishes badly. In this case, the performance of the uncontrolled (salmon) and the controlled but unconstrained (gold) processes is bad, and all the constraint designs diverge substantially from the unconstrained ones. There is still some extent of clustering of the optimal designs of the VaR (turquoise) and $\mathbb{P}$-ES (blue) constraints and between the designs of the $\mathbb{Q}$-ES (pink) and the Strict (green) constraints. Due to their (early) focus on protecting the level $C$, the constrained processes perform well and, in this case, investing in protection paid off.

\section{Conclusion} \label{S_C}

In this paper, we considered the surplus of an insurance company endowed with an initial level of risk. This risk is approximated with a Brownian motion. The insurance company can transfer some of its risk at a premium (non-cheap reinsurance). At the same premium, it can also take on even more risk and act as the reinsurer, which is optimal in some cases. This decision minimises the quadratic distance between the terminal value at a given deterministic time frame and an exogenous constant. This linear-quadratic (LQ) formulation is directly connected to a mean-variance optimisation criterion.

We considered five different constraints in the optimisation, which directly focus on solvency, and which are of direct relevance in the insurance context: (i) no constraint (as a benchmark), (ii) a strict solvency constraint that forces the surplus to terminate above an arbitrary level, (iii) a ``Value at Risk'' constraint that requires the former event to occur with minimum probability (as opposed to certainty), (iv) an ``Expected Shortfall'' constraint that requires the expected value of the shortfall (for the terminal value) under the solvency level, under the objective probabilities $\mathbb{P}$-measure, to be less than an arbitrary constant, and (v) the same as (iv), although under the risk-free probabilities $\mathbb{Q}$-measure.

The paper's main contributions are to determine the optimal terminal payoff under all five constraints, from which the optimal reinsurance designs follow. To the best of our knowledge, the last two constraints were never considered before in an insurance context, even though they are directly relevant to several regulatory frameworks worldwide. They have also been shown to be better risk measures in several ways. Furthermore, while we borrow our approach from the finance literature, our problems differ in several ways (for instance, our surplus can become negative), our results are materially different, and comparisons are of great interest.

Possible extensions of our work include considering an exponential utility function in the objective, including time-dependent parameters (such as a reinsurance premium that follows soft and hard cycles), and adding a formal constraint for the reinsurance design to be between zero and one. Clarifying the validity of the structure of the designs without approximating the risk process by a Brownian motion is the most interesting (and challenging) one, both from a mathematical and contextual point of view.

\section*{Acknowledgments}

This paper was presented at UNSW (Sydney, Australia) in November 2021, at Monash University (Melbourne, Australia) in March 2022, at the ASTIN Colloquium (online) in June 2022, and at the 25th International Congress on Insurance: Mathematics and Economics (online) in July 2022.  
The authors are grateful for constructive comments from colleagues who attended those events.

Benjamin Avanzi acknowledges support under Australian Research Council's Discovery Project (DP200101859) funding scheme.  
The views expressed herein are those of the authors and are not necessarily those of the supporting organisations. 

\section*{References}

\bibliographystyle{elsarticle-harv}
\bibliography{libraries}

\newpage

\appendix



\section{Optimal designs}

\subsection{Proof of Theorem \ref{thm.optimal.C}}
\label{app.proof.C}

\begin{proof}
Consider the optimisation problem of maximizing the function $f$ over $B$ defined by
\begin{equation}
    f(B;\omega):=-\frac{1}{2}(k-B)^2-\lambda^*_C[Z_T(\omega)B-x].
\end{equation}
with support $B\in[C,k]$.
This is a quadratic function with an axis of symmetry at the point $k-\lambda^*_CZ_T(\omega)$. Since this is smaller than $k$, we immediately conclude that the maximum must be obtained either at the point $k-\lambda^*_CZ_T(\omega)$ or, in case this point is smaller than $C$, then at $C$. Therefore we can write that
$$B^*(\omega):=\max(k-\lambda^*_CZ_T(\omega),C)$$
maximises $f(\cdot,\omega)$ on $[C,k]$. In other words, we have that
$$f(X_T^{\pi_C^*}(\omega),\omega)\geq f(X_T^\pi(\omega),\omega)$$
for any $\pi\in\Pi$.
Now, taking expectation and recalling that $\lambda^*_C$ is chosen such that $\mathbb{E}[Z_TX_T^{\pi_{\lambda^*_C}^C}]=x$, we see that
\begin{equation}
    \mathbb{E}[-\frac{1}{2}(k-X_T^{\pi^*_C})^2]=\mathbb{E}[-\frac{1}{2}(k-X_T^{\pi_{\lambda^*_C}^C})^2]\geq \mathbb{E}[-\frac{1}{2}(k-X_T^\pi)^2],
\end{equation}
which is our desired result.
\end{proof}

\subsection{Proof of Theorem \ref{thm.optimal.P}}
\label{app.proof.P}

\begin{proof}

Denote
\begin{equation}
    \gamma^*:=\frac{(C-c)^2}{2}>0
\end{equation}
and consider the optimisation problem of maximising
\begin{equation}
    f(B,\omega):=-\frac{1}{2}(k-B)^2-{\lambda^*_P}[Z_T(\omega)B-x]+\gamma^*[1(B\geq C)-(1-\varepsilon)].
\end{equation}

In the following, we verify that
\begin{equation}\label{Bstar.P}
    B^*(\omega)=\begin{cases}
    k-{\lambda^*_P} Z_T(\omega),\quad &\text{if }k-{\lambda^*_P}Z_T(\omega)\geq C\\
    C,\quad &\text{if }k-{\lambda^*_P}Z_T(\omega)\in[c^*,C)\\
    k-{\lambda^*_P} Z_T(\omega),\quad &\text{if }k-{\lambda^*_P}Z_T(\omega)<c^*
\end{cases},
\end{equation}
maximises $f(\cdot,\omega)$. We first split the entire range $(-\infty,k]$ into $(-\infty,C)$ and $[C,k]$.

On the interval $[C,k]$, it is clear that
$$f(B,\omega)\leq f(\max(k-\lambda^*_PZ_T(\omega),C),\omega).$$
Similarly, on the interval $(-\infty,C]$, we have
$$f(B,\omega)\leq \max(f(\min(k-\lambda^*_PZ_T(\omega),C),\omega),f(C,\omega)).$$
In other words, we wish to show that
\begin{equation}\label{compare.P}
    f(B^*(\omega),\omega)\geq \max(f(\min(k-\lambda^*_PZ_T(\omega),C),\omega),f(\max(C,k-\lambda^*_PZ_T(\omega),\omega)).
\end{equation}
It is sufficient to show the above for all three cases described in \eqref{Bstar.P}.

Case 1: $k-{\lambda^*_P}Z_T(\omega)\geq C$. In this case, it is easy to see that the right-hand side of \eqref{compare.P} is the same as the left-hand side.

Case 2: $k-{\lambda^*_P}Z_T(\omega)\in[c^*,C]$. In this case, the right-hand side of \eqref{compare.P} is
$$\max(f(k-{\lambda^*_P}Z_T(\omega),\omega),f(C,\omega)).$$By taking the difference between the two arguments, we see
\begin{align*}
    &f(k-{\lambda^*_P}Z_T(\omega),\omega)-f(C,\omega)
    =\frac{1}{2}(\lambda Z_T(\omega)-k+C)^2-\gamma^*>0\\
\iff~&C-(k-{\lambda^*_P} Z_T(\omega))>\sqrt{2\gamma^*}  
\iff k-{\lambda^*_P} Z_T(\omega)< C-\sqrt{2\gamma^*}=c^*.
\end{align*}
and therefore the right-hand side of \eqref{compare.P} is $f(C,\omega)$, the same as the left-hand side.

Case 3: $k-{\lambda^*_P}Z_T(\omega)\in(-\infty,c^*]$. In this case, the right hand side of \eqref{compare.P} is also $$\max(f(k-{\lambda^*_P}Z_T(\omega),\omega),f(C,\omega)).$$ However, via the same analysis as above, we see that this expression is now evaluated as $f(k-{\lambda^*_P}Z_T(\omega),\omega)=f(B^*(\omega),\omega)$.

This concludes our assertion that $B^*(\omega)$ is a maximiser of $f(\cdot,\omega)$. 
In other words, we have derived that
$$f(X_T^{\pi^*_P}(\omega),\omega)\geq f(X_T^\pi,\omega)$$
for any admissible design $\pi$. Integrating both sides and making use of the constraints, we see that
\begin{equation}
    \mathbb{E}[-\frac{1}{2}(k-X_T^{\pi^*_P})^2]\geq \mathbb{E}[-\frac{1}{2}(k-X_T^\pi)^2]+\gamma^*(\mathbb{P}[X_T^\pi\geq C]-(1-\varepsilon))\geq \mathbb{E}[-\frac{1}{2}(k-X_T^\pi)^2],
\end{equation}
which completes the proof.
\end{proof}

\subsection{Proof of Theorem \ref{thm.optimal.E}}
\label{app.proof.E}

\begin{proof}

Consider the optimisation problem of maximising
\begin{equation}
    f(B,\omega):=-\frac{1}{2}(k-B)^2-{\lambda^*_S}[Z_T(\omega)B-x]-\gamma^*_S[(C-B)1(C\geq B)-c].
\end{equation}

In the following, we verify that
\begin{equation}\label{Bstar.S}
    B^*(\omega)=\begin{cases}
    k-{\lambda^*_S} Z_T(\omega),\quad &\text{if }k-{\lambda^*_S}Z_T(\omega)\geq C,\\
    C,\quad &\text{if }k-{\lambda^*_S}Z_T(\omega)\in[C-\gamma^*_S,C),\\
    k-{\lambda^*_S} Z_T(\omega)+\gamma_S^*,\quad &\text{if }k-{\lambda^*_S}Z_T(\omega)<C-\gamma^*_S,
\end{cases}
\end{equation}
maximises $f(\cdot,\omega)$. We first split the entire range $(-\infty,k]$ into $(-\infty,C)$ and $[C,k]$.

On interval $[C,k]$, we have
$$f(B,\omega)\leq f(\max(k-\lambda^*_SZ_T(\omega),C),\omega).$$
Similarly, on interval $(-\infty,C]$, we have
$$f(B,\omega)\leq f(\min(k-\lambda^*_SZ_T(\omega)+\gamma^*_S,C),\omega).$$
In other words, we wish to show
\begin{equation}\label{compare.S}
    f(B^*(\omega),\omega)\geq \max(f(\min(k-\lambda^*_SZ_T(\omega)+\gamma^*_S,C),\omega),f(\max(C,k-\lambda^*_SZ_T(\omega),\omega)).
\end{equation}
It is sufficient to show the above for all three cases described in \eqref{Bstar.S}.

Case 1: $k-{\lambda^*_S}Z_T(\omega)\geq C$. In this case, it is easy to see that the right-hand side of \eqref{compare.S} is the same as the left-hand side.

Case 2: $k-{\lambda^*_S}Z_T(\omega)\in[C-\gamma^*_S,C]$. In this case, the right hand side of \eqref{compare.S}
$$\max(f(k-{\lambda^*_S}Z_T(\omega)+\gamma^*_S,\omega),f(C,\omega))$$
is $f(C,\omega)$ if and only if $k-{\lambda^*_S}Z_T(\omega)+\gamma^*_S>C$, which is equivalent to $k-{\lambda^*_S}Z_T(\omega)>C-\gamma^*_S$. Hence, we get \eqref{compare.S}.

Case 3: $k-{\lambda^*_S}Z_T(\omega)\in(-\infty,C-\gamma^*_S]$. Via the same analysis as above, we see that this expression is now evaluated as $f(k-{\lambda^*_S}Z_T(\omega)+\gamma^*_S,\omega)=f(B^*(\omega),\omega)$.

This concludes our assertion that $B^*(\omega)$ is a maximiser of $f(\cdot,\omega)$. 
In other words, we have derived
$$f(X_T^{\pi^*_S}(\omega),\omega)\geq f(X_T^\pi,\omega)$$
for any admissible design $\pi$. Integrating both sides and making use of the constraints, we see that
\begin{equation}
    \mathbb{E}[-\frac{1}{2}(k-X_T^{\pi^*_S})^2]\geq \mathbb{E}[-\frac{1}{2}(k-X_T^\pi)^2]-\gamma^*_S(\mathbb{E}[(C-X_T^\pi)_+]-c)\geq \mathbb{E}[-\frac{1}{2}(k-X_T^\pi)^2],
\end{equation}
which completes the proof.
\end{proof}

\subsection{Proof of Theorem \ref{thm.optimal.Q}}
\label{app.proof.Q}

\begin{proof}

Consider the optimisation problem of maximising
\begin{equation}
    f(B,\omega):=-\frac{1}{2}(k-B)^2-{\lambda^*_\mathbb{Q}}[Z_T(\omega)B-x]-(\lambda^*_\mathbb{Q}-\delta^*_\mathbb{Q})[Z_T(\omega)(C-B)1(C\geq B)-\nu].
\end{equation}

In the following, we verify 
\begin{equation}\label{Bstar.Q}
    B^*(\omega)=\begin{cases}
    k-{\lambda^*_\mathbb{Q}} Z_T(\omega),\quad &\text{if }Z_T(\omega)\leq \frac{k-C}{\lambda^*_\mathbb{Q}},\\
    C,\quad &\text{if }Z_T(\omega)\in[\frac{k-C}{\lambda^*_\mathbb{Q}},\frac{k-C}{\delta^*_\mathbb{Q}}],\\
    k-\delta^*_\mathbb{Q} Z_T(\omega),\quad &\text{if }Z_T(\omega)>\frac{k-C}{\delta^*_\mathbb{Q}},
\end{cases}
\end{equation}
maximises $f(\cdot,\omega)$. We first split the entire range $(-\infty,k]$ into $(-\infty,C)$ and $[C,k]$.

On interval $[C,k]$, we have
$$f(B,\omega)\leq f(\max(k-\lambda^*_\mathbb{Q} Z_T(\omega),C),\omega).$$
Similarly, on interval $(-\infty,C]$, we have
$$f(B,\omega)\leq f(\min(k-\delta^*_\mathbb{Q} Z_T(\omega),C),\omega).$$
In other words, we wish to show
\begin{equation}\label{compare.Q}
    f(B^*(\omega),\omega)\geq \max(f(\min(k-\delta^*_\mathbb{Q} Z_T(\omega),C),\omega),f(\max(C,k-\lambda^*_\mathbb{Q} Z_T(\omega),\omega)).
\end{equation}
It is sufficient to show the above for all three cases described in \eqref{Bstar.Q}.

Case 1: $Z_T(\omega)\leq ({k-C})/{\lambda^*_\mathbb{Q}}$. In this case, it is easy to see that the right-hand side of \eqref{compare.Q} is the same as the left-hand side.

Case 2: $Z_T(\omega)\in[({k-C})/{\lambda^*_\mathbb{Q}},({k-C})/\delta^*_\mathbb{Q}]$. In this case, the right-hand side of \eqref{compare.Q}
$$\max(f(k-\delta^*_\mathbb{Q} Z_T(\omega),\omega),f(C,\omega))$$
is $f(C,\omega)$ if and only if $k-\delta^*_\mathbb{Q} Z_T(\omega)>C$, which is equivalent to the condition defining Case 2. Hence, we get \eqref{compare.Q}.

Case 3: $Z_T(\omega)\geq ({k-C})/\delta^*_\mathbb{Q}$. Via the same analysis as above, we see that this expression is now evaluated as $f(k-\delta^*_\mathbb{Q} Z_T(\omega),\omega)=f(B^*(\omega),\omega)$.

This concludes our assertion that $B^*(\omega)$ is a maximiser of $f(\cdot,\omega)$. 
In other words, we have derived
$$f(X_T^{\pi^*_\mathbb{Q}}(\omega),\omega)\geq f(X_T^\pi,\omega)$$
for any admissible design $\pi$. Integrating both sides and making use of the constraints, we see that
\begin{equation}
    \mathbb{E}[-\frac{1}{2}(k-X_T^{\pi^*_\mathbb{Q}})^2]\geq \mathbb{E}[-\frac{1}{2}(k-X_T^\pi)^2]-(\lambda^*_\mathbb{Q}-\delta^*_\mathbb{Q})(\mathbb{E}[Z_T(C-X_T^\pi)_+]-\nu)\geq \mathbb{E}[-\frac{1}{2}(k-X_T^\pi)^2],
\end{equation}
which completes the proof.
\end{proof}

\section{Auxiliary results}

\subsection{Computation of Expectations}
Note 
$$Z_T\sim e^{-\frac{1}{2}\beta^2 T+(-\beta\sqrt{T})Z},$$
where $Z$ is a $N(0,1)$ random variable,
we have for $b>a>0$
\begin{equation}\label{app.help}
\begin{cases}
    &\mathbb{E}[Z_T1(Z_T\in[a,b])]=\Phi(\frac{\frac{1}{2}\beta^2T-\log b}{\beta \sqrt{T}})-\Phi(\frac{\frac{1}{2}\beta^2T-\log a}{\beta \sqrt{T}})\\
    &\mathbb{E}[Z_T^21(Z_T\in[a,b])]=e^{\beta^2T}\Big(\Phi(\frac{\frac{3}{2}\beta^2T-\log b}{\beta \sqrt{T}})-\Phi(\frac{\frac{3}{2}\beta^2T-\log a}{\beta \sqrt{T}})\Big)
\end{cases}.
\end{equation}

Note that at time $t$, we have to update those formulas with $\mathcal{F}_t$. In this case, $Z_t$ is known, and the increment $Z_T/Z_t$ has the same distribution as an increment $Z_{T-t}$. However, $a$ and $b$ in \eqref{app.help} need to be \emph{updated},
\begin{equation} \label{E_update}
Z_T\in[a,b] \equiv Z_t \times \left\{ \frac{Z_T}{Z_t}\in\left[\frac{a}{Z_t},\frac{b}{Z_t}\right]\right\}.
\end{equation}
The quantity in curly brackets in \eqref{E_update} can then be calculated using \eqref{app.help} and the updated boundaries $[a/Z_t,b/Z_t]$. This \emph{update} is referred to in other remarks below.

\subsection{Computation of Proportions}\label{ss.comp.pit}

We compute the general form of $\pi_t$ in this section. Denote
\begin{equation}\label{app.aux2}
\begin{cases}
    &f(t,z) = \Phi(\frac{\frac{1}{2}\beta^2t-\log z}{\beta \sqrt{t}})\\
    &g(t,z) = e^{\beta^2t}\Phi(\frac{\frac{3}{2}\beta^2t-\log z}{\beta \sqrt{t}})\\
    &h(t,z) = \frac{g(t,z)}{z}\\
    &m(t,z) = ze^{\beta^2t}
\end{cases},
\end{equation}
we see later in Section \ref{A_paths} that $X^\pi_t$ is a linear combination of $1$, $f(T-t,a_j/Z_t)$, $h(T-t,a_j/Z_t)$ and $m(T-t,Z_t)$. Our goal is to derive the Stochastic Differential Equation for $X^\pi_t$. Once this is done, $\pi_t$ can be obtained by matching coefficients. It suffices to investigate only $f$, $h$ and $m$, and there is nothing to do for the constant. In the following, we look into the terms individually.

We start with $f$ and $h$. First, observe the following PDEs are satisfied:
\begin{equation}\label{app.aux2.pde}
\begin{cases}
    &f_t = \frac{1}{2}\beta^2 z^2 f_{zz}\\
    &g_t = \beta^2 g + \frac{1}{2} \beta^2 z^2 g_{zz} -\beta^2 z g_z
\end{cases},
\end{equation}
which implies
\begin{equation}
    h_t = \frac{1}{2}\beta^2 z^2 h_{zz}
\end{equation}
which has a similar form as $f$ in \eqref{app.aux2.pde}.

Therefore, in the following, we consider $f$ and denote 
\begin{equation}
    V(t,z) = f(T-t,\frac{\alpha}{z}),    
\end{equation}
where $\alpha >0$ is a constant. Consequently, we have
\begin{equation}
    V_t = -f_t,\quad V_z = f_z\frac{-\alpha}{z^2},\quad 
    V_{zz} = f_{zz}\frac{\alpha^2}{z^4} + f_z\frac{2\alpha}{z^3},
\end{equation}
where the argument inside $f$, $(T-t,\alpha/z)$, is omitted for clarity of presentation and hence
$$
    V_t + \frac{\beta^2z^2}{2}V_{zz} = \Big(-f_t + \frac{1}{2}\beta^2(\frac{\alpha}{z})^2f_{zz}\Big) + \beta^2(\frac{\alpha}{z})f_z = \beta^2(\frac{\alpha}{z})f_z,
$$
where the terms inside the bracket are zero by the property of $f$, see \eqref{app.aux2.pde}.

From this, we see that via \rev{It\=o's Lemma}
\begin{align*}
    dV(t,Z_t) =~& V_tdt + V_zdZ_t + \frac{1}{2}V_{zz}(dZ_t)^2\\
    =~& V_t dt + \beta Z_t V_z dW_t + \frac{1}{2}\beta^2Z_t^2V_{zz}dt\\
    =~& \beta^2(\frac{\alpha}{Z_t})f_z dt - \beta\frac{\alpha}{Z_t}f_zdW_t\\
    =~& \frac{-\beta}{\sigma}\frac{\alpha}{Z_t}f_z(T-t,\frac{\alpha}{Z_t})
    \Big(
        bdt + \sigma dW_t    
    \Big),\label{abc}
\end{align*}
where we have used the definition of $\beta$, i.e. $\beta = -b/\sigma$. 

The last form of $V$ is 
\begin{equation}
    V(t,z) = m(T-t,z) = z e^{\beta^2 (T-t)}.
\end{equation}
Direct computation gives
\begin{equation}
    V_t = -\beta^2 V, \quad V_z = V/z, \quad V_{zz} = 0
\end{equation}
and therefore, via \rev{It\=o's Lemma}, we have
\begin{align*}
    dV(t,Z_t) = ~&
    -\beta^2V(t,Z_t)dt + \beta Z_t V(t,Z_t)/Z_tdW_t\\
    =~& \frac{\beta m(T-t,Z_t)}{\sigma} (b dt + \sigma dW_t).
\end{align*}

Therefore, if we write
\begin{equation}\label{xpit.linear.comb}
    X^\pi_t = V(t,Z_t) = \sum_j \alpha_j f_j(T-t,\frac{a_j}{Z_t})
            - \xi m(T-t,Z_t),
\end{equation}
where $\xi$, $\alpha_j$, $a_j$ are constants and each $f_j$ is either $f$ or $h$ (or constant), see Section \ref{A_paths} for the exact linear combination for each case. Then via \rev{It\=o's Lemma}, we have
$$
dX^\pi_t = \frac{-\beta}{\sigma}
        \Big(
            \sum_j\alpha_j\frac{a_j}{Z_t}\parD{z}f_j\Big\rvert_{(T-t,\frac{a_j}{Z_t})}
            +\xi m(T-t,Z_t)
        \Big)
                        \Big(
                            bdt + \sigma dW_t    
                        \Big).
$$
Recall
$$
dX^{\pi}_t = (1-\pi_t)(bdt+\sigma dW_t),
$$
we can conclude that 
\begin{equation}\label{sol.pit}
    \pi_t = 1+ \frac{\beta}{\sigma}\Big(
    \sum_j\alpha_j\frac{a_j}{Z_t}\parD{z}f_j\Big\rvert_{(T-t,\frac{a_j}{Z_t})}
    +\xi m(T-t,Z_t)
    \Big).
\end{equation}
\begin{remark}
    Strictly speaking, one can match the $dW_t$ term, so the complete calculation of $dV(t,Z_t)$ is unnecessary.
\end{remark}

For computation, the following are required:
\begin{equation}
    \begin{cases}
        &zf_z(t,z) = \frac{-1}{\beta\sqrt{t}} \phi(\frac{\frac{1}{2}\beta^2t-\log z}{\beta\sqrt{t}})\\
        &z h_z(t,z) = \frac{e^{\beta^2t}}{z}\Big(
            \frac{-1}{\beta\sqrt{t}} \phi(\frac{\frac{3}{2}\beta^2t-\log z}{\beta\sqrt{t}})
            - \Phi(\frac{\frac{3}{2}\beta^2t-\log z}{\beta\sqrt{t}})
        \Big)
    \end{cases}
\end{equation}
where 
$$
\phi: x \mapsto \frac{1}{\sqrt{2\pi}}e^{-\frac{1}{2}x^2}
$$
is the Normal density.

\section{Paths} \label{A_paths}
Calculation of a sample path $X_t^\pi$ at time $t$ for a design $\pi$ conditioning on the filtration $\mathscr{F}_t$ follows directly from \eqref{find.sample.path} where individual expectation can be calculated using \eqref{app.help}. This also computes the optimal design $\pi^*$ thanks to \eqref{xpit.linear.comb} and \eqref{sol.pit}.

\subsection{Unconstrained problem}
For a design $\pi\in \Pi_U$ with parameter $\lambda>0$, its path at time $t\in[0,T]$ is given by 
\begin{align*}
    X^\pi_t = ~& \mathbb{E}[\frac{Z_T}{Z_t}(k-\lambda Z_T)|\mathscr{F}_t]\\
        =~& \mathbb{E}[\frac{Z_T}{Z_t}(k-\lambda Z_t\frac{Z_T}{Z_t})|\mathscr{F}_t]\\
        =~& \mathbb{E}[\widetilde{Z}_{T-t}(k-\lambda {Z}_t \widetilde{Z}_{T-t})|{Z}_t]\\
        =~& k-\lambda Z_t \mathbb{E}[\widetilde{Z}_{T-t}^2]\\
        =~& k-\lambda Z_t e^{\beta^2(T-t)},
\end{align*}
where $\widetilde{Z}_{T-t}$ is an independent copy of $Z_{T-t}$. Note the ``$=$" notation above indicates equality in distribution, and the expectation term in the last line can be evaluated easily, e.g. with the help of \eqref{app.help}.

To compute the proportion $\pi_t$, we rewrite the above as
$$
X^\pi_t = k - \lambda m(T-t,Z_t).
$$

\subsection{Strict constraint}\label{ss.fs}
Similar to above, for a design $\pi\in \Pi_C$ with parameter $\lambda>0$, denote
$$
    \eta = \frac{k-C}{\lambda},
$$
the path of $X^\pi$ at time $t\in[0,T]$ is given by 
\begin{align*}
    X^\pi_t = ~& \mathbb{E}[\frac{Z_T}{Z_t}(k-\frac{k-C}{\eta} Z_T)1(Z_T<\eta)|\mathscr{F}_t] + \mathbb{E}[\frac{Z_T}{Z_t}C1(Z_T>\eta)|\mathscr{F}_t]\\
        =~& \mathbb{E}[\frac{Z_T}{Z_t}(k-\frac{k-C}{\eta} Z_t \frac{Z_T}{Z_t})1(\frac{Z_T}{Z_t}<\frac{\eta}{Z_t})|\mathscr{F}_t] + \mathbb{E}[\frac{Z_T}{Z_t}C1(\frac{Z_T}{Z_t}>\frac{\eta}{Z_t})|\mathscr{F}_t]\\
        =~& \mathbb{E}[\widetilde{Z}_{T-t}(k-\frac{k-C}{\eta/Z_t} \widetilde{Z}_{T-t})1(\widetilde{Z}_{T-t}<\frac{\eta}{Z_t})|Z_t] + \mathbb{E}[\widetilde{Z}_{T-t}C 1(\widetilde{Z}_{T-t}>\frac{\eta}{Z_t})|Z_t]\\
        =~&
            k\mathbb{E}[\widetilde{Z}_{T-t};\widetilde{Z}_{T-t}<\frac{\eta}{Z_t}|Z_t]
            - \frac{k-C}{\eta/Z_t} \mathbb{E}[\widetilde{Z}_{T-t}^2;\widetilde{Z}_{T-t}<\frac{\eta}{Z_t}|Z_t]
          + C \mathbb{E}[\widetilde{Z}_{T-t};\widetilde{Z}_{T-t}>\frac{\eta}{Z_t}|Z_t],
\end{align*}
where we have used a shorthand notation
$$
    \mathbb{E}[\cdot;A] = \mathbb{E}[\cdot 1(A)].
$$
Again, treating $Z_t$ as a constant and applying \eqref{app.help}, the above expression can be evaluated with ease.
\begin{remark}
    Recall by definition $\eta = (k-C)/\lambda$, the update of the $\eta$ to $\eta/Z_t$ amounts to update $\lambda$ to $\lambda Z_t$.
\end{remark}

To compute the proportion $\pi_t$, we rewrite the above as 
\begin{align*}
    X^\pi_t &=
            k\mathbb{E}[\widetilde{Z}_{T-t};\widetilde{Z}_{T-t}<\frac{\eta}{Z_t}|Z_t]
            - \frac{k-C}{\eta/Z_t} \mathbb{E}[\widetilde{Z}_{T-t}^2;\widetilde{Z}_{T-t}<\frac{\eta}{Z_t}|Z_t]
          + C \mathbb{E}[\widetilde{Z}_{T-t};\widetilde{Z}_{T-t}>\frac{\eta}{Z_t}|Z_t] \\
          &= k f(T-t, \frac{\eta}{Z_t}) - (k-C) h(T-t, \frac{\eta}{Z_t}) + C(1 - f(T-t, \frac{\eta}{Z_t}))
\end{align*}

\subsection{Probability (VaR) constraint}
Similar to above, for a design $\pi\in \Pi_P$ with parameters $0<g_1\leq g_2$, its path at time $t\in[0,T]$ is given by 
\begin{align*}
    X^\pi_t = ~& \mathbb{E}[\frac{Z_T}{Z_t}(k-\frac{k-C}{g_1} Z_T)1(Z_T\notin[g_1,g_2])|\mathscr{F}_t] + 
                  \mathbb{E}[\frac{Z_T}{Z_t}C1(Z_T\in[g_1,g_2])|\mathscr{F}_t]\\
            =~& \mathbb{E}[(\frac{Z_T}{Z_t}k - \frac{k-C}{g_1/Z_t}(\frac{Z_T}{Z_t})^2)1(\frac{Z_T}{Z_t}\notin[\frac{g_1}{Z_t},\frac{g_2}{Z_t}])|\mathscr{F}_t] + \mathbb{E}[\frac{Z_T}{Z_t}C1(\frac{Z_T}{Z_t}\in[\frac{g_1}{Z_t},\frac{g_2}{Z_t}])|\mathscr{F}_t]\\
            =~&
                k \mathbb{E}[\widetilde{Z}_{T-t};\widetilde{Z}_{T-t}\notin[\frac{g_1}{Z_t},\frac{g_2}{Z_t}]|Z_t]-
                \frac{k-C}{g_1/Z_t}\mathbb{E}[\widetilde{Z}_{T-t}^2;\widetilde{Z}_{T-t}\notin[\frac{g_1}{Z_t},\frac{g_2}{Z_t}]|Z_t]\\
            &    + C \mathbb{E}[\widetilde{Z}_{T-t};\widetilde{Z}_{T-t}\in[\frac{g_1}{Z_t},\frac{g_2}{Z_t}]|Z_t].
\end{align*}
\begin{remark}
    Recall by definition $g_1 = (k-C)/\lambda$ and $g_2 = (k-c)/\lambda$, the update of $g_1$ to $g_1/Z_t$ and $g_2$ to $g_2/Z_t$ amounts to update $\lambda$ to $\lambda Z_t$.
\end{remark}

To compute the proportion $\pi_t$, we rewrite the above as
\begin{align*}
    X^\pi_t = ~&
    k (1 - f(T-t, \frac{g_2}{Z_t}) + f(T-t, \frac{g_1}{Z_t}))\\
    & +(k-C) ( \frac{g_2}{g_1}h(T-t, \frac{g_2}{Z_t}) - h(T-t, \frac{g_1}{Z_t})) - \frac{k-C}{g_1}m(T-t,Z_t) \\
    &+ C (f(T-t, \frac{g_2}{Z_t}) - f(T-t, \frac{g_1}{Z_t}))
\end{align*}

\subsection{Expected shortfall constraint under $\mathbb{P}$}
Similar to above, for a design $\pi\in \Pi_E$ with parameters $0<h_1\leq h_2$ (and therefore $\gamma = (h_2-h_1)(k-C)/h_1$), its path at time $t\in[0,T]$ is given by 
\begin{align*}
    X^\pi_t = ~& \mathbb{E}[\frac{Z_T}{Z_t}(k-\frac{k-C}{h_1} Z_T)1(Z_T\notin[h_1,h_2])|\mathscr{F}_t] + 
                  \mathbb{E}[\frac{Z_T}{Z_t}C1(Z_T\in[h_1,h_2])|\mathscr{F}_t]\\
                & + \gamma \mathbb{E}[\frac{Z_T}{Z_t}1(Z_T>h_2)|\mathscr{F}_t]\\
            =~& \mathbb{E}[\frac{Z_T}{Z_t}(k-\frac{k-C}{h_1/Z_t} \frac{Z_T}{Z_t})1(\frac{Z_T}{Z_t}\notin[\frac{h_1}{Z_t},\frac{h_2}{Z_t}])|\mathscr{F}_t] + 
                  \mathbb{E}[\frac{Z_T}{Z_t}C1(\frac{Z_T}{Z_t}\in[\frac{h_1}{Z_t},\frac{h_2}{Z_t}])|\mathscr{F}_t]\\
                & + \gamma \mathbb{E}[\frac{Z_T}{Z_t}1(\frac{Z_T}{Z_t}>\frac{h_2}{Z_t})|\mathscr{F}_t]\\
           =~&
                k \mathbb{E}[\widetilde{Z}_{T-t};\widetilde{Z}_{T-t}\notin[\frac{h_1}{Z_t},\frac{h_2}{Z_t}]|Z_t]-
                \frac{k-C}{h_1/Z_t}\mathbb{E}[\widetilde{Z}_{T-t}^2;\widetilde{Z}_{T-t}\notin[\frac{h_1}{Z_t},\frac{h_2}{Z_t}]|Z_t]\\
            &    + C \mathbb{E}[\widetilde{Z}_{T-t};\widetilde{Z}_{T-t}\in[\frac{h_1}{Z_t},\frac{h_2}{Z_t}]|Z_t] 
                + \gamma \mathbb{E}[\widetilde{Z}_{T-t};\widetilde{Z}_{T-t}>\frac{h_2}{Z_t}|Z_t].
\end{align*}
\begin{remark}
    Similar to above, the update of $h_1$ to $h_1/Z_t$ and $h_2$ to $h_2/Z_t$ amounts to update $\lambda$ to $\lambda Z_t$.
\end{remark}

To compute the proportion $\pi_t$, we rewrite the above as 
\begin{align*}
    X^\pi_t =~&
            k (1 - f(T-t, \frac{h_2}{Z_t}) + f(T-t, \frac{h_1}{Z_t})) \\
            &+ (k-C) ( \frac{h_2}{h_1} h(T-t, \frac{h_2}{Z_t})
            - h(T-t, \frac{h_1}{Z_t})) - \frac{k-C}{h_1}m(T-t, Z_t) \\
            &   + C ( f(T-t, \frac{h_2}{Z_t}) - f(T-t, \frac{h_1}{Z_t})) \\
            &+ \gamma (1 - f(T-t, \frac{h_2}{Z_t}))
\end{align*}

\subsection{Expected shortfall constraint under $\mathbb{Q}$}\label{ss.qs}
Similar to above, for a design $\pi\in \Pi_\mathbb{Q}$ with parameters $0<\delta\leq \lambda$, its path at time $t\in[0,T]$ is given by 
\begin{align*}
    X^\pi_t = ~& \mathbb{E}[\frac{Z_T}{Z_t}(k-\lambda Z_T)1(Z_T<\frac{k-C}{\lambda})|\mathscr{F}_t] + 
                  \mathbb{E}[\frac{Z_T}{Z_t}C1(Z_T\in[\frac{k-C}{\lambda},\frac{k-C}{\delta}])|\mathscr{F}_t]\\
                & + \mathbb{E}[\frac{Z_T}{Z_t}(k-\delta Z_T)1(Z_T>\frac{k-C}{\delta})|\mathscr{F}_t]\\
            =~& \mathbb{E}[\frac{Z_T}{Z_t}(k-\lambda Z_t \frac{Z_T}{Z_t})1(\frac{Z_T}{Z_t}<\frac{k-C}{\lambda Z_t})|\mathscr{F}_t] + 
                  \mathbb{E}[\frac{Z_T}{Z_t}C1(\frac{Z_T}{Z_t}\in[\frac{k-C}{\lambda Z_t},\frac{k-C}{\delta Z_t}])|\mathscr{F}_t]\\
                & + \mathbb{E}[\frac{Z_T}{Z_t}(k-\delta Z_t \frac{Z_T}{Z_t})1(\frac{Z_T}{Z_t}>\frac{k-C}{\delta Z_t})|\mathscr{F}_t]\\
            =~& k \mathbb{E}[\widetilde{Z}_{T-t};\widetilde{Z}_{T-t}<\frac{k-C}{\lambda Z_t}|Z_t]-
                \lambda Z_t\mathbb{E}[\widetilde{Z}_{T-t}^2;\widetilde{Z}_{T-t}<\frac{k-C}{\lambda Z_t}|Z_t] \\
                & + C \mathbb{E}[\widetilde{Z}_{T-t};\widetilde{Z}_{T-t}\in[\frac{k-C}{\lambda Z_t},\frac{k-C}{\delta Z_t}]|Z_t] \\
                & + k \mathbb{E}[\widetilde{Z}_{T-t};\widetilde{Z}_{T-t}>\frac{k-C}{\delta Z_t}|Z_t]-
                \delta Z_t\mathbb{E}[\widetilde{Z}_{T-t}^2;\widetilde{Z}_{T-t}>\frac{k-C}{\delta Z_t}|Z_t].
\end{align*}

\begin{remark}
    The update amounts to update $\lambda$ to $\lambda Z_t$ and $\delta$ to $\delta Z_t$.
\end{remark}

To compute the proportion $\pi_t$, we rewrite the above as 
\begin{align*}
    X^\pi_t 
    =~& k \mathbb{E}[\widetilde{Z}_{T-t};\widetilde{Z}_{T-t}<\frac{k-C}{\lambda Z_t}|Z_t]
     + k \mathbb{E}[\widetilde{Z}_{T-t};\widetilde{Z}_{T-t}>\frac{k-C}{\delta Z_t}|Z_t]\\
     & + C \mathbb{E}[\widetilde{Z}_{T-t};\widetilde{Z}_{T-t}\in[\frac{k-C}{\lambda Z_t},\frac{k-C}{\delta Z_t}]|Z_t] \\
      &- (k-C)\frac{\lambda}{k-C} Z_t\mathbb{E}[\widetilde{Z}_{T-t}^2;\widetilde{Z}_{T-t}<\frac{k-C}{\lambda Z_t}|Z_t] \\
      &-(k-C)\frac{\delta}{k-C} Z_t\mathbb{E}[\widetilde{Z}_{T-t}^2;\widetilde{Z}_{T-t}>\frac{k-C}{\delta Z_t}|Z_t]\\
    =~& k ( f(T-t, \frac{k-C}{\lambda Z_t}) + 1-
        f(T-t,\frac{k-C}{\delta Z_t}|Z_t]))\\
     & + C (f(T-t,\frac{k-C}{\delta Z_t}|Z_t]) - f(T-t,\frac{k-C}{\lambda Z_t}|Z_t]))\\
      &- (k-C)h(T-t,\frac{k-C}{\lambda Z_t}) \\
      &+(k-C)h(T-t, \frac{k-C}{\delta Z_t}) - \delta m(T-t,Z_t)
\end{align*}

\end{document}